\documentclass[reqno]{amsart}
\usepackage{csquotes}
\usepackage{hyperref}
\usepackage{amsmath}
\usepackage{amssymb}
\usepackage{mathrsfs}
\usepackage{accents}
\usepackage{graphicx}

\usepackage{tabularx,colortbl}

\usepackage{tikz}
\usetikzlibrary{calc}
\usepackage{xcolor}
\usetikzlibrary{arrows.meta}
\usepackage{amsthm}
\usepackage{float}
\usepackage{enumitem}
\usepackage[english]{babel}
\usepackage[font=footnotesize, labelfont=bf]{ caption}
\usepackage{ esint }
\usepackage{ dsfont }

\newcommand{\eg}{{\it e.g.}}
\newcommand{\ie}{{\it i.e.}}

\theoremstyle{plain}
\begingroup
\newtheorem{theorem}{Theorem}[section]
\newtheorem{lemma}[theorem]{Lemma}
\newtheorem{proposition}[theorem]{Proposition}

\endgroup

\newcommand{\step}[1]{{ \underline{\itshape Step #1}}.}

\theoremstyle{definition}
\begingroup
\newtheorem{definition}[theorem]{Definition}

\endgroup
\theoremstyle{remark}
\newtheorem{remark}[theorem]{Remark}
\renewcommand{\tilde}{\widetilde}
\newcommand{\sm}{\setminus}

\renewcommand{\d}{ \mathrm{d}}

\DeclareMathOperator{\supp}{supp}

\newcommand{\pvint}{\, \mathrm{p.v.} \! \int}

\numberwithin{equation}{section}
\newcommand{\N}{\mathbb{N}}
\newcommand{\Z}{\mathbb{Z}}
\newcommand{\R}{\mathbb{R}}

\newcommand{\de}{\partial}
\newcommand{\e}{\varepsilon}
\newcommand{\x}{{\times}}
\renewcommand{\hat}{\widehat}

\newcommand{\T}{\mathbb{T}}
\newcommand{\W}{\mathcal{W}}

\usepackage{mathtools}
\mathtoolsset{showonlyrefs}

\newcommand{\myequation}{\begin{equation}}
  \newcommand{\myendequation}{\end{equation}}
  \let\[\myequation
  \let\]\myendequation

  \title[Jump discontinuities for periodic peridynamic waves]{Dissolution and nucleation of jump discontinuities for one-dimensional periodic peridynamic waves}

  \author[G. M. Coclite]{G. M. Coclite}
  \address[Giuseppe Maria Coclite]{\newline
     Dipartimento di Meccanica, Matematica e Management, Politecnico di Bari,
    Via E.~Orabona 4, I--70125 Bari, Italy.}
  \email[]{giuseppemaria.coclite@poliba.it}

  \author[S. Dipierro]{S. Dipierro}
  \address[Serena Dipierro]{\newline Department of Mathematics and Statistics,
  University of Western Australia, 35 Stirling Highway, WA6009 Crawley, Australia}
  \email[]{serena.dipierro@uwa.edu.au}

  \author[F. Maddalena]{F. Maddalena}
  \address[Francesco Maddalena]{\newline
   Dipartimento di Meccanica, Matematica e Management, Politecnico di Bari,
   Via E.~Orabona 4, I--70125 Bari, Italy.}
  \email[]{francesco.maddalena@poliba.it}

  \author[G. Orlando]{G. Orlando}
  \address[Gianluca Orlando]{\newline
   Dipartimento di Meccanica, Matematica e Management, Politecnico di Bari,
   Via E.~Orabona 4, I--70125 Bari, Italy.}
  \email[]{gianluca.orlando@poliba.it}

  \author[E. Valdinoci]{E. Valdinoci}
  \address[Enrico Valdinoci]{\newline Department of Mathematics and Statistics,
  University of Western Australia, 35 Stirling Highway, WA6009 Crawley, Australia.}
  \email[]{enrico.valdinoci@uwa.edu.au}

\usepackage[
backend=biber,
style=numeric-verb,
citestyle=numeric-comp,
sorting=nty,
maxbibnames=99
]{biblatex}
\begin{document}

\begin{abstract}
    We analyze solutions to a one-dimensional periodic linear model for peridynamic waves, focusing on the setting where the natural energy space allows for spatial discontinuities. Exploiting the high-frequency asymptotics of the peridynamic dispersion relation, we construct an explicit solution whose initial displacement has a single jump discontinuity but becomes continuous at every positive time. By time reversibility, this yields a continuous initial displacement whose evolution develops a jump at any prescribed time and is continuous at every other time. We further quantify the formation of this singularity through a sharp H\"older-norm blow-up estimate obtained via a Littlewood-Paley characterization of H\"older spaces. Finally, by superposing suitably weighted solutions, we construct an evolution exhibiting jump discontinuities on a dense set of times.
\end{abstract}

\maketitle

\setcounter{tocdepth}{1}
\tableofcontents

\section{Introduction}

Understanding crack initiation and its propagation in Solid Mechanics is a constant challenge in a satisfactory rational description of the physical evolution of real materials.
A successful path towards this goal is certainly given by the variational approach to fracture, pursued in the last years, stemming from the works~\cite{FraMar98} and subsequently developed in~\cite{DMToa02, Cha03, FraLar03,DMFraToa05}.
This line of thought relies on the assumption that crack evolution is ruled by equilibrium and energy balance processes, placing the problem in the more general framework of energetic solutions~\cite{MieThe99, Mie05, MieRou15}.
Coping with the problem of dealing with discontinuities under weak notions of differentiability has stimulated the development of new fine tools in contemporary Calculus of Variations, as shown in more recent contributions (see, \eg,~\cite{GiaPon06, DMZan07, DMLaz10, CagToa11, FriSol18, CriLazOrl18, BonConIur21, CriFri26}), showing a continually growing interest in this approach over the years.
However, the dynamical nature of the creation and evolution of fracture lies out of the scope of this picture, although oscillations play a crucial role in experimentally observed behaviour of real materials.
Progresses in extending this approach to elastodynamics has been made in~\cite{DMLar11, DMLarToa16, DMLazNar16, Cap17, DMLuc17, LazNar19, DMToa19, Cap20, DMDL20, CapLucTas20, CapSap20, RivNar21, CiaDM21, Cia23, LazMolRivSol23, CapCarSap24}, though a complete mathematical theory seems, at the moment, out of reach.

An alternative approach to overcome the problem of singularities arising in Continuum Mechanics relies on freeing oneself from modelling strains through spatial derivatives of the displacement field, and instead considering difference quotients, thereby leading to nonlocal equations.
This paradigm is at the core of the peridynamics model, proposed in~\cite{Sil00}.
Among the commonly recognized advantages of these models, one can pinpoint the low regularity required for the displacement field, allowing one, at least from a computational perspective, to naturally include singularities in solutions~\cite{SilAsk05, HaBob10, BobZha15, CocFanLopMadPel20, LopPel21, LopPel22, LopPel22-2, CocCocMadPol24} and study crack nucleation and propagation~\cite{SilWecAsk10, NiaCheBob21}.
We refer to the handbook~\cite{BobFosGeuSil16} for a more recent reference covering both theory and computational applications.

The aim of this paper is to frame these problems in a rigorous setting, investigating the emergence of spatial discontinuities occurring in the simplest scenario allowed by a precise mathematical analysis.
Specifically, we consider the one-dimensional torus $\T^1$ and solutions $u \colon [0,+\infty) \times \T^1 \to \R$ to the following Cauchy problem:
\begin{equation} \label{eqintro:peridynamics}
  \begin{cases}
    \de_{tt} u(t,x) - K[u(t,\cdot)](x) = 0 \, , & \text{for } t > 0 \text{ and } x \in \T^1 \, , \\
    u(0,x) = u_0(x) \, , \quad \de_t u(0,x) = v_0(x) \, , & \text{for } x \in \T^1 \, ,
  \end{cases}
\end{equation}
where $u_0$ and $v_0$ are given initial data and $K$ is the peridynamic operator given by
\begin{equation} \label{eqintro:K}
    K[u](x) = - \pvint_{\R} \chi_\delta(y) \frac{u(x) - u(x-y)}{|y|^{1+2\alpha}} \d y \, , \quad \text{for every } x \in \T^1 \, ,
\end{equation}
where $\delta > 0$ and $\alpha \in (0,1)$ are fixed parameters and $\chi_\delta$ is supported in $[-\delta,\delta]$ and smooth in $[-\delta,\delta]$, see Section~\ref{sec:model} for details.\footnote{The simplest example of such a function $\chi_\delta$ is the indicator function of the interval $[-\delta,\delta]$. The reader may keep this example in mind throughout the paper.}
The equation considered in~\eqref{eqintro:peridynamics} is the balance of linear momentum in a periodic medium, where internal forces are expressed in terms of the displacement field $u$ through nonlocal interactions ruled by $K$.
The parameter $\delta > 0$ is the so-called peridynamic horizon, representing the spatial scale of the interaction range and measuring the size of the nonlocality.
The parameter $\alpha \in (0,1)$ is a fractional exponent, affecting the strength of the nonlocal interactions and will play a crucial role in the following analysis.

An inspection of~\eqref{eqintro:peridynamics} and~\eqref{eqintro:K} suggests the educated guess that the fractional Sobolev space $H^\alpha(\T^1)$ should be the natural choice as a functional space to look for solutions, see Subsection~\ref{subsec:energy-space}.
In particular, when $\alpha \in (0, \frac{1}{2})$, discontinuous functions $u$ can belong to $H^\alpha(\T^1)$.
Hence, the question we address is the following: Can a continuous initial datum $u_0 \in H^\alpha(\T^1)$ evolve into a solution $u(t,\cdot) \in H^\alpha(\T^1)$ exhibiting a space jump discontinuity at some time $t = t_1 > 0$?

We provide a positive answer to this question in the following theorem, proven in Section~\ref{sec:nucleation}.
Items (1)--(3) give the qualitative result, while (4) delivers its quantitative counterpart.

\begin{theorem} \label{thm:nucleation}
  Let $\alpha \in (0,\frac{1}{2})$.
  Let $t_1 > 0$ be fixed.
  There exist initial data $u_0 \in H^\alpha(\T^1)$ and $v_0 \in L^2(\T^1)$ such that $u_0$ is a continuous function and the unique solution $u \in C([0,+\infty);H^\alpha(\T^1)) \cap C^1([0,+\infty);L^2(\T^1))$ to the Cauchy problem~\eqref{eqintro:peridynamics} with initial data $u_0$ and $v_0$ satisfies the following conditions: 
  \begin{enumerate}
    \item $u(t_1,\cdot)$ is continuous for $x \neq 0$;
    \item $u(t_1,\cdot)$ has a jump discontinuity at $x = 0$;
    \item $u(t,\cdot)$ is a continuous function for every $t \neq t_1$;
    \item for every $s \in (0,\frac{\alpha}{2}]$ there exist $\varepsilon_0 > 0$ and $C > 0$ such that
  \begin{equation*}
      \frac{1}{C} |t-t_1|^{-\frac{s}{\alpha}} \leq \| u(t, \cdot) \|_{C^{0,s}(\T^1)} \leq C |t-t_1|^{-\frac{s}{\alpha}} \, , \quad \text{ for } 0 < |t-t_1| < \varepsilon_0 \, .
  \end{equation*}
    \end{enumerate}
\end{theorem}

To explain the result, a brief comparison with the behaviour of solutions to the classical wave equation is in order.
When internal forces are modelled by local interactions, say $\Delta u$, D'Alembert's formula rules out the nucleation of discontinuities starting from continuous initial displacements.
This principle is a manifestation of the dispersion relation of the classical wave equation, mediated by the Fourier multiplier
\begin{equation*}
    (\widehat{\Delta u})_k = - 4 \pi^2 |k|^2 \hat{u}_k \, , \quad \text{for } k \in \Z \, .
\end{equation*}
In particular, the phase velocity is independent of the frequency, \ie, all modes travel with the same speed.

On the contrary, the result in Theorem~\ref{thm:nucleation} hinges on the dispersive nature of the peridynamics equation, due to the Fourier multiplier
\begin{equation*}
    (\hat{K[u]})_k = - \omega^2(k) \hat{u}_k \, , \quad \text{for } k \in \Z \, ,
\end{equation*}
where $\omega(k)$ has an explicit formula given in Lemma~\ref{lem:K_fourier} below.
The behaviour of $\omega(k)$ is substantially different from that of the dispersion relation of the Laplacian, see Figure~\ref{fig:dispersion}.
In spite of the nontrivial shape of $\omega(k)$, for the purpose of studying the creation of singularities, what truly matters is interference between high-frequency waves, governed by the behaviour of $\omega(k)$ for large $k$.
One sees (\eg~\cite{WeckSillAsk08, CocDipFanMadRomVal21, CocDipFanMadVal22, CocCocFanMad23}) that
\begin{equation*}
    \omega(k) \sim \gamma |k|^\alpha \quad \text{for } |k| \to +\infty \, ,
\end{equation*}
for some $\gamma > 0$, implying a phase velocity $\sim \frac{\omega(k)}{2 \pi |k|}$ decaying to zero at high frequencies.
This allows for modes travelling at different speeds, and, in turn, for constructive interference between high-frequency waves, which then results in the formation of a jump discontinuity.

\begin{figure}[H]
    \includegraphics[scale=0.4]{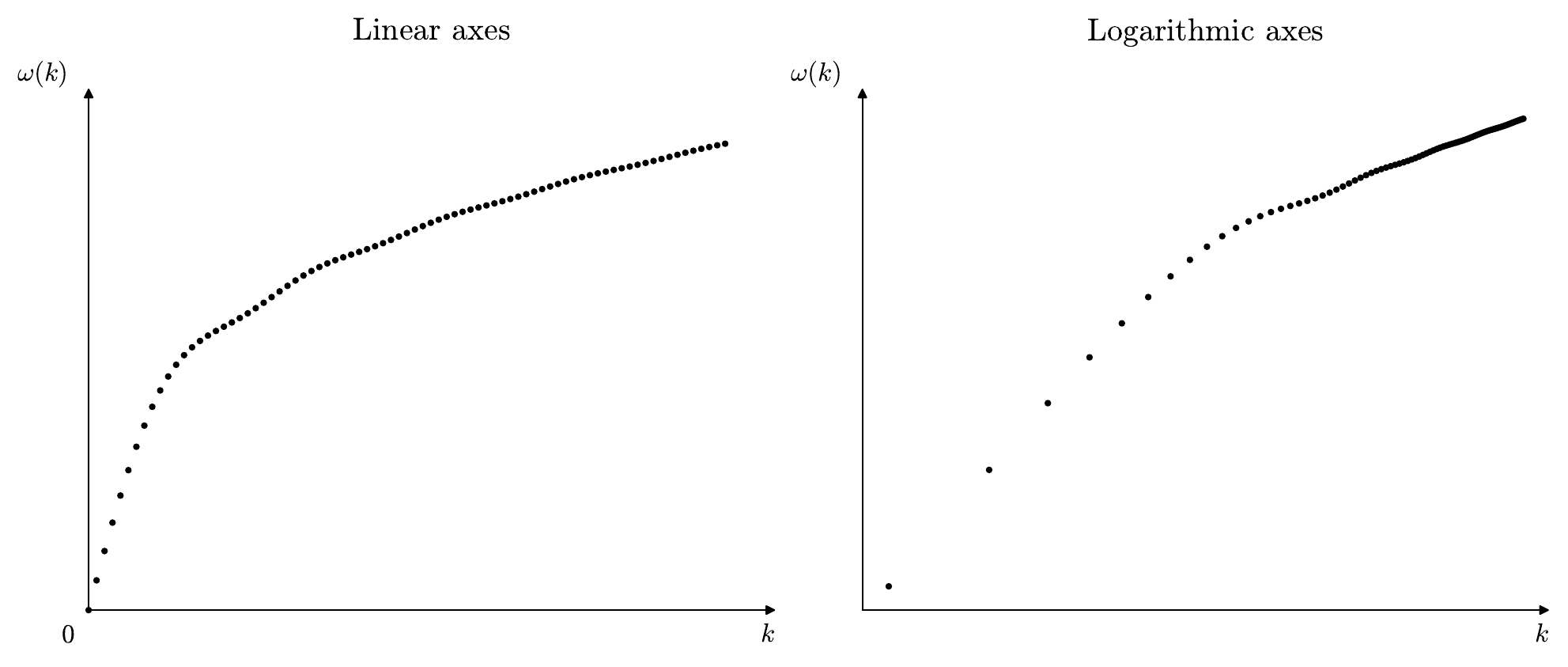}
    \caption{Plots of the dispersion relation $\omega(k)$: On the left in linear axes, on the right in logarithmic axes, depicting the power law $\omega(k) \sim \gamma |k|^\alpha$ for $k$ large. In this plot: $\chi_\delta$ is the indicator function of $[-\delta,\delta]$; $\delta = 0.06$; $\alpha = \frac{1}{4}$; frequencies plotted in the range $1 \leq k \leq 80$.}
    \label{fig:dispersion}
\end{figure}

Thanks to the time-reversal invariance of the peridynamics equation, we obtain the proof of Theorem~\ref{thm:nucleation} through an explicit example of a discontinuous initial displacement evolving into a continuous displacement at every positive time.
This example is provided in Theorem~\ref{thm:dissolution}, where we study precisely the peridynamic evolution with initial data
\[
u_0(x) = \frac{1}{2} - x \, , \quad  v_0(x) = 0 \, , \quad \text{for } x \in [0,1) \, ,
\]
illustrated in Figure~\ref{fig:simulation_zoom}.

\begin{figure}[H]
    \begin{center}
        \begin{tabular}{c c c}
            \hline \\
            $t=0.0$ & $t=0.2$ & $t=0.4$ \\
            \includegraphics[scale=1]{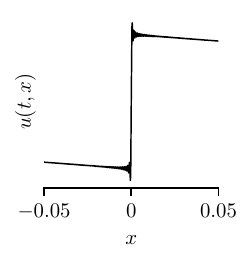} &
            \includegraphics[scale=1]{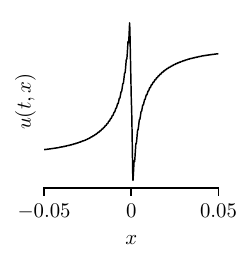} &
            \includegraphics[scale=1]{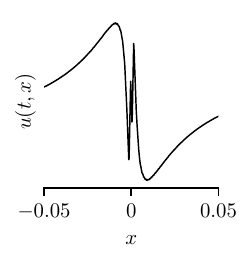} \\
            \hline \\
            $t=0.6$ & $t=0.8$ & $t=1.0$ \\
            \includegraphics[scale=1]{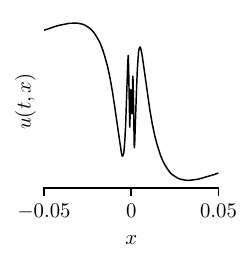} &
            \includegraphics[scale=1]{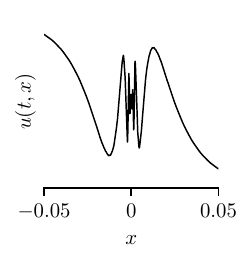} &
            \includegraphics[scale=1]{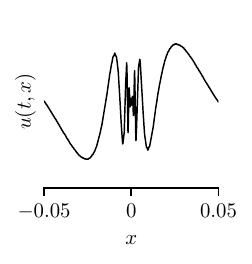} \\
            \hline \\
            $t=1.2$ & $t=1.4$ & $t=1.6$ \\
            \includegraphics[scale=1]{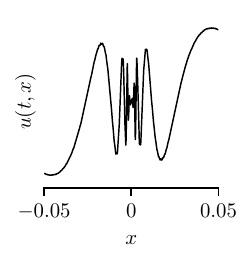} &
            \includegraphics[scale=1]{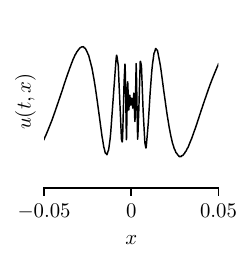} &
            \includegraphics[scale=1]{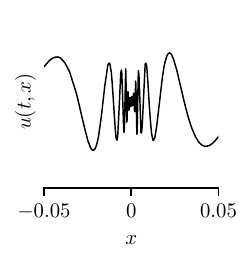} \\
            \hline
        \end{tabular}
    \end{center}
    \caption{Numerical simulation depicting the behaviour of the proposed solution close to the origin, where the jump discontinuity is dissolved. For each labeled $t$, each panel shows $u(t,\cdot)$ plotted in a neighborhood $[-0.05, 0.05]$ of the jump point $0$. In this plot: $\chi_\delta$ is the indicator function of $[-\delta,\delta]$; $\delta = 1$; $\alpha = 0.3$; the simulation is obtained by truncation of Fourier series at 1000 modes.}
    \label{fig:simulation_zoom}
\end{figure}

At the core of our analysis sit Van der Corput-type estimates for the uniform convergence of oscillatory series
\[
\sum_{k \in \Z \setminus \{0\}} \frac{1}{k} e^{i 2\pi k x + i \omega(k) t} \, ,
\]
developed in detail in Section~\ref{subsec:technical}.
These estimates also play the key role in the quantitative H\"older estimates in Theorem~\ref{thm:nucleation}-(4), emerging from precise $B^s_{\infty,\infty}$-Besov estimates.

Finally, the nucleation of a discontinuity obtained for a single time is extended to a dense set of times in Theorem~\ref{thm:dense_set}.

We conclude this introduction by mentioning that the phenomenon at the heart of our results places this paper in a broader literature interested in understanding the interplay between dispersion and spatial discontinuities of periodic solutions~\cite{BerGol88, Osk92, BerKle96, KapRod99, BerMarSch01, Olv10, ErdTzi13, dLHVeg14}.

\section{Basic notation and preliminary results}

\subsection{Basic notation}
We let $C$ denote a positive constant, which may vary from line to line.

\subsection{Fourier series}
Let $\T^1 = \R/\Z$ be the $1$-dimensional torus.
Throughout the paper, functions $u \colon \T^1 \to \R$ will be tacitly extended to periodic functions defined on $\R$.

The Fourier coefficients of a function $u \in L^1(\T^1)$ are defined by
\[
\hat{u}_k = \int_{\T^1} u(x) e^{-i 2 \pi k x} \d x \, , \quad \text{for every } k \in \Z \, .
\]

\subsection{Fractional Sobolev spaces}
Let $\alpha \in (0,1)$.
The \emph{fractional Sobolev space} $H^\alpha(\T^1)$ is the space
\[
H^\alpha(\T^1) := \Big\{ u \in L^2(\T^1) \text{ such that } \sum_{k \in \Z} (1 + |k|^2)^\alpha |\hat{u}_k|^2 < +\infty \Big\} \, .
\]
It is a Hilbert space endowed with the inner product
\[
\langle u, v \rangle_{H^\alpha(\T^1)} := \sum_{k \in \Z} (1 + |k|^2)^\alpha \hat{u}_k \overline{\hat{v}_k} \, , \quad \text{for every } u, v \in H^\alpha(\T^1) \, .
\]
We refer to~\cite{DNPalVal12} for more details about fractional Sobolev spaces.

\subsection{ H\"older space $C^{0,s}$ and Besov space $B^s_{\infty,\infty}$} \label{subsec:holder_besov}

We recall here the relation between H\"older spaces and Besov spaces, which will be useful later in the paper in Section~\ref{sec:holder}.

\subsubsection{H\"older space $C^{0,s}$}
Given $s \in (0,1)$, the space of H\"older continuous functions with exponent $s$ is defined as
\begin{equation*}
    C^{0,s}(\T^1) := \Big\{ u \in L^\infty(\T^1) \text{ such that } \sup_{ \substack{x, y \in \R \\ x \neq y} } \frac{ |u(x) - u(y)| }{ |x - y|^s } < +\infty \Big\} \, ,
\end{equation*}
endowed with the norm
\begin{equation*}
    \| u \|_{C^{0,s}(\T^1)} := \|u\|_{L^\infty(\T^1)} + [u]_{C^{0,s}(\T^1)} \, ,
\end{equation*}
where
\begin{equation*}
    [u]_{C^{0,s}(\T^1)} := \sup_{ \substack{x, y \in \R \\ x \neq y} } \frac{ |u(x) - u(y)| }{ |x - y|^s } \, .
\end{equation*}

Next, we recall  a characterization of H\"older continuity expressed in terms of the distribution of the $L^\infty$ norm across frequency scales.
We provide the details of the definition of the Besov norm, since they will be useful later in the paper in Section~\ref{sec:holder}.

\tikzset{declare function = {bump(\x) = exp(1)*exp(1/((2*\x-1)*(2*\x-1) - 1));}}

\subsubsection{Littlewood-Paley decomposition}
We introduce the Littlewood-Paley decomposition of a function, see, \eg, \cite[Chapter~VI, Paragraph~4.1]{Ste93}.
We start by fixing a partition of unity in the Fourier space.
Let $\psi_{-1} \in C^{1,1}_c(\R)$ be such that $\psi_{-1}(\xi) = 1$ for $|\xi| \leq \frac{1}{2}$ and $\psi_{-1}(\xi) = 0$ for $|\xi| \geq 1$.
In the transition region, we choose the transition function $\psi_{-1}(\xi) = e^{1+\frac{1}{(2|\xi| - 1)^2-1}}$ for $\frac{1}{2} < |\xi| < 1$.
For $j \geq 0$, set $\psi_j(\xi) := \psi_{-1}(2^{-j-1} \xi) - \psi_{-1}(2^{-j} \xi)$ and observe that $\psi_j$ is supported where $2^{j-1} \leq |\xi| \leq 2^{j+1}$, see Figure~\ref{fig:partition_of_unity}.
Moreover,
\begin{equation*}
    \psi_{-1}(\xi) + \sum_{j = 0}^{+\infty} \psi_j(\xi) = 1 \, , \quad \text{for every } \xi \in \R \, ,
\end{equation*}
where the series is a finite sum for every $\xi \in \R$.

\begin{figure}[ht]
    {\centering
        \begin{tikzpicture}[scale=1]
        \draw (0,0) -- (8,0);

        \draw (0,-0.1) -- (0, 0.1);
        \draw (0,-0.1) node[anchor=north] {$0$};

        \draw (0.5,-0.1) -- (0.5, 0.1);
        \draw (0.5,-0.1) node[anchor=north] {$\frac{1}{2}$};

        \draw (1,-0.1) -- (1, 0.1);
        \draw (1,-0.1) node[anchor=north] {$1$};

        \draw (2,-0.1) -- (2, 0.1);
        \draw (2,-0.1) node[anchor=north] {$2$};

        \draw (4,-0.1) -- (4, 0.1);
        \draw (4,-0.1) node[anchor=north] {$4$};

        \draw (8,-0.1) -- (8, 0.1);
        \draw (8,-0.1) node[anchor=north] {$8$};

        \draw[thick, domain=0:0.5, smooth, variable=\x] plot (\x, 1);
        \draw[thick, domain=0.5:1, smooth, variable=\x] plot (\x, {bump(\x)});

        \draw (0, 1.1) node[anchor=south] {$\psi_{-1}$};

        \draw[thick, domain=0.5:1, smooth, variable=\x] plot (\x, {1 - bump(\x)});
        \draw[thick, domain=1:2, smooth, variable=\x] plot (\x, {bump(1/2 * \x)});

        \draw (1, 1.1) node[anchor=south] {$\psi_{0}$};

        \draw[thick, domain=1:2, smooth, variable=\x] plot (\x, {1 - bump(1/2 * \x)});
        \draw[thick, domain=2:4, smooth, variable=\x] plot (\x, {bump(1/4 * \x)});

        \draw (2, 1.1) node[anchor=south] {$\psi_{1}$};

        \draw[thick, domain=2:4, smooth, variable=\x] plot (\x, {1 - bump(1/4 * \x)});
        \draw[thick, domain=4:8, smooth, variable=\x] plot (\x, {bump(1/8 * \x)});

        \draw (4, 1.1) node[anchor=south] {$\psi_{2}$};

    \end{tikzpicture}
    }
    \caption{Partition of unity in the Fourier space.}
    \label{fig:partition_of_unity}
\end{figure}
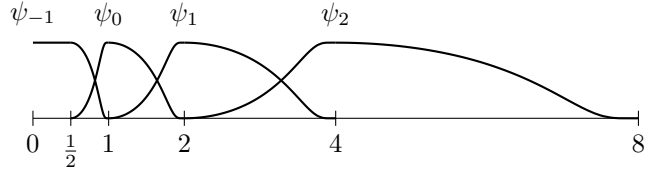

We deduce here a precise estimate from below for $\psi_j$ that will be useful later in Section~\ref{sec:holder}.

\begin{lemma} \label{lem:psi_lower_bound}
   Let $j \geq 0$ and $k \in \Z$.
Then
\begin{equation*}
    \psi_j(|k|) \geq \frac{1}{2} \iff \Big\lceil 2^{j-1} \Big( 1 + \Big( \frac{\log(2)}{1 + \log(2)} \Big)^\frac{1}{2} \Big) \Big\rceil \leq |k| \leq \Big\lfloor 2^{j} \Big( 1 + \Big( \frac{\log(2)}{1 + \log(2)} \Big)^\frac{1}{2} \Big) \Big\rfloor \, .
\end{equation*}
\end{lemma}
\begin{proof}
If $|k| = 2^j$ the inequality is trivial, since $\psi_j(2^j) = 1$.
We discuss the two cases obtained by splitting at $|k| = 2^j$.
Indeed, for $k \in \Z$ such that
\begin{equation*}
    2^{j-1} < 2^{j-1} \Big( 1 + \Big( \frac{\log(2)}{1 + \log(2)} \Big)^\frac{1}{2} \Big) \leq |k| < 2^j \, ,
\end{equation*}
we have that
\begin{equation*}
  \begin{split}
      1 + \Big( \frac{\log(2)}{1 + \log(2)} \Big)^\frac{1}{2} \leq 2^{-j+1}|k| & \iff -\frac{1}{1+\log(2)} \leq (2^{-j+1}|k| - 1)^2 - 1 \\
                                                                               & \iff 1 + \frac{1}{(2^{-j+1}|k| - 1)^2 - 1} \leq -\log(2) \\
                                                                               & \iff \psi_{-1}(2^{-j}|k|) = e^{1+\frac{1}{( 2^{-j+1} |k| -1)^2 - 1}} \leq \frac{1}{2} \, ,
  \end{split}
\end{equation*}
whence,
\begin{equation*}
    \psi_j(|k|) = \psi_{-1}(2^{-j-1}|k|) - \psi_{-1}(2^{-j}|k|) = 1 - \psi_{-1}(2^{-j}|k|) \geq \frac{1}{2} \, .
\end{equation*}
Analogously, for $k \in \Z$ such that
\begin{equation*}
    2^j < |k| \leq 2^j \Big( 1 + \Big( \frac{\log(2)}{1 + \log(2)} \Big)^\frac{1}{2} \Big) < 2^{j+1} \, ,
\end{equation*}
we have that
\begin{equation*}
  \begin{split}
      2^{-j}|k| \leq 1 + \Big( \frac{\log(2)}{1 + \log(2)} \Big)^\frac{1}{2} & \iff - \frac{1}{1+\log(2)} \geq (2^{-j}|k| - 1)^2 - 1 \\
                                                                             & \iff 1 + \frac{1}{(2^{-j}|k| - 1)^2 - 1} \geq -\log(2) \\
                                                                             & \iff \psi_{-1}(2^{-j-1}|k|) = e^{1+\frac{1}{(  2^{-j} |k| - 1)^2 -1 }} \geq \frac{1}{2} \, ,
  \end{split}
\end{equation*}
whence
\begin{equation*}
    \psi_j(|k|) = \psi_{-1}(2^{-j-1}|k|) \geq \frac{1}{2} \, .
\end{equation*}
This concludes the proof.
\end{proof}

The partition of unity is exploited to define the Littlewood-Paley decomposition of a function.

Let $p \in (1,+\infty)$.
Given a function $u \in L^p(\T^1)$ with Fourier series
\begin{equation*}
    u(x) = \sum_{k \in \Z} \hat{u}_k e^{i 2 \pi k x} \, , 
\end{equation*}
for every $j \geq 0$, we introduce the dyadic block $P_j u$ defined by
\begin{equation*}
    P_j u(x) := \sum_{k \in \Z} \psi_j(|k|) \hat{u}_k e^{i 2 \pi k x} = \sum_{2^{j-1} \leq |k| < 2^{j+1}} \psi_j(|k|) \hat{u}_k e^{i 2 \pi k x} \, , \quad  x \in \T^1 \, ,
\end{equation*}
and
\begin{equation*}
    P_{-1} u(x) := \sum_{|k| < 2^{-1}} \psi_{-1}(|k|) \hat{u}_k e^{i 2 \pi k x} = \hat{u}_0 \, , \quad  x \in \T^1 \, .
\end{equation*}
We observe that
\begin{equation} \label{eq:LP_decomposition}
    u(x) = \sum_{j \geq -1} P_j u(x) \, , 
\end{equation}
where the series converges in $L^p(\T^1)$.

\subsubsection{Besov space $B^s_{\infty,\infty}$}
Let $s \in (0,1)$.
Let $u \in L^\infty(\T^1)$.
We consider the Besov space
\begin{equation*}
    B^s_{\infty,\infty}(\T^1) := \Big\{ u \in L^\infty(\T^1) \text{ such that } \sup_{j \geq -1} 2^{s j} \| P_j u \|_{L^\infty(\T^1)} < +\infty \Big\}
\end{equation*}
endowed with the norm
\begin{equation*}
    \| u \|_{B^s_{\infty,\infty}(\T^1)} := \sup_{j \geq -1} 2^{s j} \| P_j u \|_{L^\infty(\T^1)} \, .
\end{equation*}

\subsubsection{Equivalence between norms in $C^{0,s}$ and $B^s_{\infty,\infty}$}

In the next lemma, we recall the equivalence between the H\"older norm $\| \cdot \|_{C^{0,s}(\T^1)}$ and the Besov norm $\| \cdot \|_{B^s_{\infty,\infty}(\T^1)}$, which is a classical result in harmonic analysis.

\begin{lemma} \label{lem:holder_besov}
Let $s \in (0,1)$.
Let $u \in L^\infty(\T^1)$.
Then there exists a constant $C > 0$ such that:
\begin{itemize}
    \item if $u \in C^{0,s}(\T^1)$, then $u \in B^s_{\infty,\infty}(\T^1)$ and $\| u \|_{B^s_{\infty,\infty}(\T^1)} \leq C \| u \|_{C^{0,s}(\T^1)}$;
    \item if $u \in B^s_{\infty,\infty}(\T^1)$, then $u$ admits a  representative in $C^{0,s}(\T^1)$ (not relabeled) and $\| u \|_{C^{0,s}(\T^1)} \leq C \| u \|_{B^s_{\infty,\infty}(\T^1)}$.
\end{itemize}
\end{lemma}
\begin{proof}
    The interested reader will find the proof in Appendix~\ref{app:holder_besov}.
\end{proof}

\section{Description of the model, basic results, and well-posedness} \label{sec:model}

\subsection{The linear peridynamics model}
We describe here in detail the linear peridynamics model studied in this paper. The original model is an integro-differential equation whose unknown is a vector field (the displacement field).
Here, we consider the one-dimensional periodic case.
Restricting the model to the one-dimensional case is convenient to simplify the exposition of the explicit example of nucleation and dissolution of jump discontinuities.
The spatial periodicity assumption is chosen for the sake of simplicity. Similar constructions can be done in the case of solutions defined on the whole real line.

\smallskip

The problem is finding solutions $u \colon (0,+\infty) \times \T^1 \to \R$ to the following Cauchy problem:
\[ \label{eq:peridynamics}
  \begin{cases}
    \de_{tt} u(t,x) - K[u(t,\cdot)](x) = 0 \, , & \text{for } t > 0 \text{ and } x \in \T^1 \, , \\
    u(0,x) = u_0(x) \, , \quad \de_t u(0,x) = v_0(x) \, , & \text{for } x \in \T^1 \, ,
  \end{cases}
\]
where $u_0$ and $v_0$ are given initial data and $K$ is the peridynamic operator defined precisely in Subsection~\ref{sec:K} below.

\subsection{The peridynamic operator} \label{sec:K}
In this subsection we define precisely the peridynamic operator and we list some of its properties.

Let $\delta > 0$ be a parameter (the peridynamic horizon) and let $\chi_\delta$ be a function such that
\[
\chi_\delta(y) = \chi\Big(\frac{|y|}{\delta}\Big) \, , \quad \text{for every } y \in \R \, ,
\]
where $\chi \in L^\infty(\R)$ satisfies:
\[ \label{eq:chi_properties}
0 \leq \chi \leq 1 \, , \quad \supp(\chi) \subset [-1,1] \, , \quad \chi(s) = 1 \text{ for } |s| < \frac{1}{2} \, , \quad \text{and} \quad \chi \in C^\infty([-1,1]) \, .
\]
Let $\alpha \in (0,1)$ be the exponent of the nonlocal interaction.
In this subsection, we provide the definitions for $\alpha \in (0,1)$, but the main results of the paper hold true for $\alpha \in (0,\frac{1}{2})$.

\begin{definition}
The \emph{peridynamic operator} $K$ is defined for any $u \in C^\infty(\T^1)$ by
\[
K[u](x) = - \pvint_{\R} \chi_\delta(y) \frac{u(x) - u(x-y)}{|y|^{1+2\alpha}} \d y \, , \quad \text{for every } x \in \T^1 \, ,
\]
where $u$ is extended by periodicity and $\pvint$ denotes the principal value integral
\[
    \pvint_{\R} \chi_\delta(y) \frac{u(x) - u(x-y)}{|y|^{1+2\alpha}}  \d y = \lim_{\e \to 0^+} \int_{\{\e \leq |y|\}} \chi_\delta(y) \frac{u(x) - u(x-y)}{|y|^{1+2\alpha}}  \d y \, .
\]
\end{definition}

We start by recasting the definition of the peridynamic operator in terms of a classical integral, which is more convenient for later computations.
The precise statement is given in the next lemma.

\begin{lemma} \label{lem:K_alternative}
Let $u \in C^\infty(\T^1)$.
Then
\[ \label{eq:K_alternative}
K[u](x) = \frac{1}{2} \int_{\R} \chi_\delta(y) \frac{u(x+y) + u(x-y) - 2 u(x)}{|y|^{1+2\alpha}} \d y \, .
\]
Moreover, $K[u] \in L^\infty(\T^1)$.
\end{lemma}
\begin{proof}
  Let $u \in C^\infty(\T^1)$.
  By the symmetry of $\chi_\delta$, a change of variables yields
  \[
  \begin{split}
    K[u](x) & = - \pvint_{\R} \chi_\delta(y) \frac{u(x) - u(x-y)}{|y|^{1+2\alpha}} \d y \\
            & = - \pvint_{\R} \chi_\delta(y) \frac{u(x) - u(x+y)}{|y|^{1+2\alpha}} \d y \, ,
  \end{split}
  \]
  whence
  \[
  K[u](x) = \frac{1}{2} \pvint_{\R} \chi_\delta(y) \frac{u(x+y) + u(x-y) - 2 u(x)}{|y|^{1+2\alpha}} \d y \, .
  \]
  The latter is, in fact, an integral in the classical sense.
  Indeed, by Taylor expansion, we have that
  \[
  u(x+y) + u(x-y) - 2u(x) = \rho(y) y^2  \, , \quad \|\rho \|_{L^\infty} \leq C \, ,
  \]
  from which we deduce the estimate
  \[
  \Big| \chi_\delta(y) \frac{u(x+y) + u(x-y) - 2 u(x)}{|y|^{1+2\alpha}} \Big| \leq C \frac{1}{|y|^{2\alpha - 1}} \, , \quad \text{on } \{|y| \leq \delta\} \, ,
  \]
  which is integrable for $\alpha < 1$ and allows us to write~\eqref{eq:K_alternative}.
  Moreover,
  \[
  |K[u](x)| \leq C \int_{\{|y| \leq \delta\}} \frac{1}{|y|^{2\alpha - 1}} \d y < +\infty \, , \quad \text{for every } x \in \T^1 \, .
  \]
  This completes the proof of the lemma.
\end{proof}

\subsection{The dispersion relation}
Next, we derive the representation of the peridynamic operator as a Fourier multiplier operator.
This representation will be useful for deriving the dispersion relation of the linear peridynamics equation.

\begin{lemma} \label{lem:K_fourier}
  Let $u \in C^\infty(\T^1)$. Then
  \[
      (\hat{K[u]})_k = - \omega^2(k) \hat{u}_k \, , \quad \text{for every } k \in \Z \, ,
  \]
  where
  \[
    \omega(k) := \Big( \int_{\R} \chi_\delta(y) \frac{1 - \cos(2\pi k y)}{|y|^{1+2\alpha}} \d y  \Big)^{\frac{1}{2}} \, , \quad \text{for every } k \in \Z \, .
  \]
  Moreover,
  \[ \label{eq:omega_limit}
  \lim_{k \to \infty} \frac{\omega(k)}{|k|^{\alpha}} = \gamma := \Big( \int_{\R} \frac{1 - \cos(2\pi z)}{|z|^{1+2\alpha}} \d z \Big)^{\frac{1}{2}} \in (0,+\infty) \, .
  \]
\end{lemma}
\begin{proof}
  Arguing as in the proof of Lemma~\ref{lem:K_alternative}, we have that
  \[
  \chi_\delta(y) \frac{u(x+y) + u(x-y) - 2 u(x)}{|y|^{1+2\alpha}} e^{-i 2 \pi k x}  \in L^1(\T^1 \times [-\delta,\delta]) \, ,
  \]
  hence by Lemma~\ref{lem:K_alternative} and by Fubini's theorem, we compute the Fourier coefficients of $K[u]$ via the formula
  \[
  \begin{split}
    \hat{K[u]}_k & = \int_{\T^1} K[u](x) e^{-i 2 \pi k x} \d x \\
    & =  \frac{1}{2}  \int_{\T^1} \Big( \int_{\R} \chi_\delta(y) \frac{u(x+y) + u(x-y) - 2 u(x)}{|y|^{1+2\alpha}} e^{-i 2 \pi k x} \d y \Big) \d x  \\
    & = \frac{1}{2} \int_{\R} \frac{\chi_\delta(y)}{|y|^{1+2\alpha}} \Big( \int_{\T^1} \Big( u(x+y) + u(x-y) - 2 u(x) \Big) e^{-i 2 \pi k x} \d x \Big) \d y \\
    & = \frac{1}{2} \Big( \int_{\R} \chi_\delta(y) \frac{e^{i 2 \pi k y}  + e^{-i 2 \pi k y}  - 2 }{|y|^{1+2\alpha}} \d y \Big) \hat{u}_k \\
    & = - \Big( \int_{\R} \chi_\delta(y) \frac{1-\cos(2 \pi k y)}{|y|^{1+2\alpha}} \d y \Big) \hat{u}_k \, .
  \end{split}
  \]

  To prove the limit in~\eqref{eq:omega_limit}, we change variables $y = \frac{z}{|k|}$, obtaining that
  \[
  \begin{split}
    \omega(k) & = \Big( \int_{\R} \chi_\delta(y) \frac{1 - \cos(2\pi k y)}{|y|^{1+2\alpha}} \d y  \Big)^{\frac{1}{2}} = |k|^{\alpha} \Big( \int_{\R} \chi_\delta\Big(\frac{z}{|k|}\Big) \frac{1 - \cos(2\pi z)}{|z|^{1+2\alpha}} \d z  \Big)^{\frac{1}{2}} \, .
  \end{split}
  \]
  Then~\eqref{eq:omega_limit} follows by the Dominated Convergence Theorem.
\end{proof}

\begin{remark} \label{rem:omega_bounds}
    By Lemma~\ref{lem:K_fourier}, we have, in particular, that there exist constants $C_1, C_2 > 0$ (depending on $\alpha$ and $\delta$) such that
  \[
  C_1 |k|^{\alpha} \leq \omega(k) \leq C_2 |k|^{\alpha} \, , \quad \text{for every } k \in \Z \setminus \{0\} \, .
  \]
\end{remark}

\begin{remark}
  By Lemma~\ref{lem:K_fourier}, we have that
  \[
  u \in H^s(\T^1) \implies K[u] \in H^{s-2\alpha}(\T^1) \, .
  \]
  In particular, if $u \in C^\infty(\T^1)$, then $K[u] \in C^\infty(\T^1)$.
\end{remark}

For the main result of the paper, it is crucial to derive a precise estimate on the derivative of the symbol $\omega$, estimating its mismatch with the derivative of the symbol $|\xi|^\alpha$ (which is, up to a normalization constant, the symbol of the fractional Laplacian $(-\Delta)^\frac{\alpha}{2}$).
We extend the definition of $\omega$ to the whole real line by setting
\[ \label{eq:omega_cont}
\omega(\xi) := \Big( \int_{\R} \chi_\delta(y) \frac{1 - \cos(2\pi \xi y)}{|y|^{1+2\alpha}} \, \d y \Big)^\frac{1}{2} \, , \quad \text{for } \xi \in \R \, .
\]
The result is the following.

\begin{lemma} \label{lem:omega_prime}
  Let $\alpha \in (0,1)$.
  Let $\omega$ be as in~\eqref{eq:omega_cont} and let $\gamma$ be as in~\eqref{eq:omega_limit}.
  Then $\omega$ is differentiable in $\R \setminus \{0\}$ and
  \[ \label{eq:omega_prime_asymptotic}
      |\omega'(\xi) - \alpha \gamma \mathrm{sign}(\xi)|\xi|^{\alpha - 1}| \leq C |\xi|^{-\alpha - 1} \, , \quad \text{for every } \xi \in \R \setminus \{0\} \, .
  \]
\end{lemma}
\begin{proof}
  Since $\omega$ is even, it is enough to prove the statement for $\xi > 0$.
  By the change of variable $z = \xi y$, we have that, by the evenness of the integrand,
  \[
  \omega^2(\xi) = 2 \int_{0}^{\delta} \chi_\delta(y) \frac{1 - \cos(2\pi \xi y)}{|y|^{1+2\alpha}} \, \d y = \xi^{2\alpha} \ 2 \int_{0}^{\delta \xi} \chi_\delta\Big(\frac{z}{\xi}\Big) \frac{1 - \cos(2\pi z)}{|z|^{1+2\alpha}} \, \d z = \xi^{2\alpha} I(\xi) \, ,
  \]
  where we set
  \[
  I(\xi) := 2 \int_{0}^{\delta \xi} \chi_\delta\Big(\frac{z}{\xi}\Big) \frac{1 - \cos(2\pi z)}{|z|^{1+2\alpha}} \, \d z \, , \quad \text{for } \xi > 0 \, .
  \]
  Note that $I(\xi) \geq 0$ for every $\xi > 0$.

  We now establish some estimates on $I$ and $I'$.
  First of all, by~\eqref{eq:chi_properties}, we obtain the lower bound for $\xi \geq \frac{1}{\delta}$:
  \[ \label{eq:I_lower_bound}
  I(\xi) \geq 2 \int_{0}^{\delta \xi/2} \chi_\delta\Big(\frac{z}{\xi}\Big) \frac{1 - \cos(2\pi z)}{|z|^{1+2\alpha}} \, \d z = 2 \int_{0}^{\delta \xi/2}  \frac{1 - \cos(2\pi z)}{|z|^{1+2\alpha}} \, \d z \geq 2 \int_0^{1/2} \frac{1 - \cos(2\pi z)}{|z|^{1+2\alpha}} \, \d z =: c\, ,
  \]
  with $c > 0$.

  Moreover, by~\eqref{eq:chi_properties} we observe that
  \[
  \begin{split}
    |I(\xi) - \gamma^2| & = 2 \Big| \int_0^{\delta \xi} \chi_\delta\Big(\frac{z}{\xi}\Big) \frac{1 - \cos(2\pi z)}{|z|^{1+2\alpha}} \, \d z - \int_0^{+\infty} \frac{1 - \cos(2\pi z)}{|z|^{1+2\alpha}} \, \d z \Big| \\
    & = 2 \Big| \int_{\delta \xi/2}^{+\infty} \Big(1 - \chi_\delta\Big(\frac{z}{\xi}\Big)\Big) \frac{1 - \cos(2\pi z)}{|z|^{1+2\alpha}} \, \d z \Big| \leq 4 \int_{\delta \xi/2}^{+\infty} \frac{1}{|z|^{1+2\alpha}} \, \d z \leq C \xi^{-2\alpha} \, ,
  \end{split}
  \]
  for every $\xi > 0$.
  In particular, we get that,
  \[\label{eq:I_minus_gamma}
  |I(\xi)^\frac{1}{2} - \gamma| = \frac{|I(\xi) - \gamma^2|}{I(\xi)^{\frac{1}{2}} + \gamma} \leq \frac{C \xi^{-2\alpha}}{\gamma} \leq C \xi^{-2\alpha}
  \]
  for every $\xi > 0$.

  Finally, by~\eqref{eq:chi_properties}, we have that
  \[
  \begin{split}
    I'(\xi) & = 2 \delta \chi_\delta\Big(\frac{\delta \xi}{\xi}\Big) \frac{1 - \cos(2\pi \delta \xi)}{\delta^{1+2\alpha} \xi^{1+2\alpha}} - 2 \int_0^{\delta \xi} \chi_\delta'\Big(\frac{z}{\xi}\Big) \frac{z}{\xi^2} \frac{1 - \cos(2\pi z)}{|z|^{1+2\alpha}} \, \d z \\
            & = \xi^{-1-2\alpha} \frac{2 \chi(1)}{\delta^{2\alpha}} (1 - \cos(2\pi \delta \xi))  - \xi^{-2} \ 2 \int_{\frac{\delta \xi}{2}}^{\delta \xi} \chi_\delta'\Big(\frac{z}{\xi}\Big) \frac{1 - \cos(2\pi z)}{|z|^{2\alpha}} \, \d z \, ,
  \end{split}
  \]
  hence
  \[ \label{eq:I_prime_bound}
      |I'(\xi)| \leq C \xi^{-1-2\alpha} + C \xi^{-2} \int_{\frac{\delta \xi}{2}}^{\delta \xi} \frac{1}{|z|^{2\alpha}} \, \d z \leq C \xi^{-1-2\alpha} + C \xi^{-2} \xi^{1-2\alpha} = C \xi^{-1-2\alpha} \, ,
  \]
  for every $\xi > 0$.

  Now, since $\omega(\xi) = \xi^\alpha I(\xi)^\frac{1}{2}$, we observe that
  \[
  \omega'(\xi) = \alpha \xi^{\alpha - 1} I(\xi)^{\frac{1}{2}} + \xi^\alpha \frac{I'(\xi)}{2 I(\xi)^{\frac{1}{2}}} \, , \quad \text{for } \xi > 0 \, ,
  \]
  from which we deduce, putting together the estimates~\eqref{eq:I_lower_bound},~\eqref{eq:I_minus_gamma} and~\eqref{eq:I_prime_bound}, that
  \[
  \begin{split}
    |\omega'(\xi) - \alpha \gamma \xi^{\alpha - 1}| \leq \alpha \xi^{\alpha - 1} |I(\xi)^{\frac{1}{2}} - \gamma| + \xi^\alpha \frac{|I'(\xi)|}{2 I(\xi)^{\frac{1}{2}}} \leq C \xi^{\alpha - 1} \xi^{-2\alpha} + C \xi^\alpha \xi^{-1-2\alpha} = C \xi^{-\alpha - 1} \, ,
  \end{split}
  \]
  for every $\xi \geq \frac{1}{\delta}$.

  If $\xi \in (0,\frac{1}{\delta})$, we obtain the estimate (possibly modifying the constant $C$) by observing that $|\omega'(\xi) - \alpha \gamma \xi^{\alpha - 1}|$ is bounded by a constant times $1+\xi^{\alpha-1}$. This concludes the proof of~\eqref{eq:omega_prime_asymptotic} and thus of the lemma.
\end{proof}

\subsection{Energy space} \label{subsec:energy-space}
In this subsection we define and characterize the natural energy space to which solutions to the peridynamics equation~\eqref{eq:peridynamics} belong.

\vspace{1em}

We define the \emph{energy space} associated with the peridynamic operator $K$ as follows:
\[
\W^\alpha_\delta (\T^1) := \Big\{ u \in L^2(\T^1) \text{ s.t. } \frac{1}{2} \int_{\T^1} \int_{\R} \chi_\delta(y) \frac{|u(x+y) - u(x)|^2}{|y|^{1+2\alpha}} \, \d y \, \d x < +\infty \Big\} \, ,
\]
endowed with the norm
\[
\|u\|_{\W^\alpha_\delta(\T^1)} = \big( \|u\|_{L^2(\T^1)}^2 + [u]_{\W^\alpha_\delta(\T^1)}^2 \big)^{\frac{1}{2}} \, ,
\]
where
\[
[u]_{\W^\alpha_\delta(\T^1)} = \Big( \frac{1}{2} \int_{\T^1} \int_{\R} \chi_\delta(y) \frac{|u(x+y) - u(x)|^2}{|y|^{1+2\alpha}} \, \d y \, \d x \Big)^{\frac{1}{2}} \, .
\]

In the next proposition, we show that the energy space $\W^\alpha_\delta(\T^1)$ is equivalent to the fractional Sobolev space $H^{\alpha}(\T^1)$.

\begin{proposition} \label{prop:W_equiv_H}
  Let $\delta > 0$ and let $\alpha \in (0,1)$.
  Then the energy space $\W^\alpha_\delta(\T^1)$ is equivalent to the fractional Sobolev space $H^{\alpha}(\T^1)$, that is, there exists a constant $C > 0$ (depending on $\alpha$ and $\delta$) such that
  \begin{align}
    \|u\|_{\W^\alpha_\delta(\T^1)} & \leq C \|u\|_{H^{\alpha}(\T^1)} \, , \quad \text{for every } u \in H^{\alpha}(\T^1) \, ,  \label{eq:W_equiv_H_1} \\
    \|u\|_{H^{\alpha}(\T^1)} & \leq C \|u\|_{\W^\alpha_\delta(\T^1)} \, , \quad \text{for every } u \in \W^\alpha_\delta(\T^1) \, . \label{eq:W_equiv_H_2}
  \end{align}
\end{proposition}
\begin{proof}
  Let $u \in C^\infty(\T^1)$.
  First of all, we show that
  \[ \label{eq:Kuu_uW}
  - \int_{\T^1} K[u](x) u(x) \, \d x = [u]_{\W^\alpha_\delta(\T^1)}^2 \, .
  \]
  Indeed, we have that
  \[
  \begin{split}
    - \int_{\T^1} K[u](x) u(x) \, \d x & = \int_{\T^1} \Big( \pvint_{\R} \chi_\delta(y) \frac{u(x) - u(x-y)}{|y|^{1+2\alpha}} \, \d y \Big) u(x) \d x \\
    & = \int_{\T^1} \lim_{\e \to 0^+} \Big( \int_{\{\e \leq |y|\}} \chi_\delta(y) \frac{u(x) - u(x-y)}{|y|^{1+2\alpha}} \d y \Big) u(x) \, \d x \, .
  \end{split}
  \]
  Note that, as in the proof of Lemma~\ref{lem:K_alternative}, we have that
  \[
  \begin{split}
    & \Big| \int_{\{\e \leq |y| \}} \chi_\delta(y) \frac{u(x) - u(x-y)}{|y|^{1+2\alpha}} \, \d y \Big| = \frac{1}{2} \Big|\int_{\{\e \leq |y| \}} \chi_\delta(y) \frac{u(x+y) + u(x-y) - 2 u(x)}{|y|^{1+2\alpha}} \, \d y \Big| \\
    & \leq C \int_{\{\e \leq |y| \leq \delta\}} \frac{\|u''\|_{L^\infty}}{|y|^{2\alpha - 1}} \, \d y \leq C  \, .
  \end{split}
  \]
  Hence, by the Dominated Convergence Theorem and the Fubini Theorem, we deduce that
  \[
  \begin{split}
    & - \int_{\T^1} K[u](x) u(x) \d x =  \lim_{\e \to 0^+} \int_{\{\e \leq |y|\}} \frac{\chi_\delta(y)}{|y|^{1+2\alpha}} \Big( \int_{\T^1} (u(x) - u(x-y)) u(x) \, \d x \Big) \, \d y \\
    & \quad = \lim_{\e \to 0^+} \int_{\{\e \leq |y| \}} \frac{\chi_\delta(y)}{|y|^{1+2\alpha}} \frac{1}{2}\Big( \int_{\T^1} (u(x) - u(x-y)) u(x) \d x - \int_{\T^1} (u(x) - u(x+y)) u(x+y) \, \d x \Big) \, \d y \\
    & \quad = \frac{1}{2}  \int_{\T^1} \lim_{\e \to 0^+}  \int_{\{\e \leq |y| \}} \chi_\delta(y) \Big( \frac{(u(x) - u(x-y)) u(x)}{|y|^{1+2\alpha}} - \frac{(u(x) - u(x+y)) u(x+y)}{|y|^{1+2\alpha}} \Big) \, \d y \, \d x \\
    & \quad = \frac{1}{2}  \int_{\T^1} \lim_{\e \to 0^+}  \int_{\{\e \leq |y| \}} \chi_\delta(y) \Big( \frac{(u(x) - u(x-y)) u(x)}{|y|^{1+2\alpha}} - \frac{(u(x) - u(x-y)) u(x-y)}{|y|^{1+2\alpha}} \Big) \, \d y \, \d x \\
    & \quad = \frac{1}{2}  \int_{\T^1} \lim_{\e \to 0^+}  \int_{\{\e \leq |y| \}} \chi_\delta(y) \frac{|u(x) - u(x-y)|^2}{|y|^{1+2\alpha}} \, \d y \, \d x \\
    & \quad = \frac{1}{2}  \int_{\T^1} \int_{\R} \chi_\delta(y) \frac{|u(x) - u(x-y)|^2}{|y|^{1+2\alpha}} \, \d y \, \d x = [u]_{\W^\alpha_\delta(\T^1)}^2 \, ,
  \end{split}
  \]
  which proves~\eqref{eq:Kuu_uW}.

  By Parseval's identity, by~\eqref{eq:Kuu_uW}, and by Lemma~\ref{lem:K_fourier} we deduce that
  \[
  [u]_{\W^\alpha_\delta(\T^1)}^2 = - \sum_{k \in \Z} \hat{K[u]}_k \overline{\hat u_k} = \sum_{k \in \Z} \omega^2(k) |\hat{u}_k|^2 \, ,
  \]
  whence
  \[
  \|u\|_{\W^\alpha_\delta(\T^1)}^2 = \sum_{k \in \Z} (1 + \omega^2(k)) |\hat{u}_k|^2 \, .
  \]
  By Remark~\ref{rem:omega_bounds}, we have that
  \[
  (1 + |k|^2)^\alpha \leq C(1 + \omega^2(k))  \quad \text{and} \quad (1 + \omega^2(k)) \leq C (1 + |k|^2)^\alpha \, ,
  \]
  for every $k \in \Z$ and for some constant $C > 0$ (depending on $\alpha$ and $\delta$).
  It directly follows that~\eqref{eq:W_equiv_H_1} and~\eqref{eq:W_equiv_H_2} hold true for $u \in C^\infty(\T^1)$.
  We extend the result to the whole spaces $H^{\alpha}(\T^1)$ and $\W^\alpha_\delta(\T^1)$ by a density argument, as explained next.

  Let us start with the proof of~\eqref{eq:W_equiv_H_1}.
  Let $u \in H^{\alpha}(\T^1)$ and let $u^n := \sum_{|k| \leq n} \hat{u}_k e^{i 2 \pi k x} \in C^\infty(\T^1)$.
  Then
  \[
  \|u - u^n\|_{H^{\alpha}(\T^1)}^2 = \sum_{|k| > n} (1 + |k|^2)^\alpha |\hat{u}_k|^2 \to 0 \, , \quad \text{as } n \to +\infty \, ,
  \]
  since $\sum_{k \in \Z}(1 + |k|^2)^\alpha |\hat{u}_k|^2 < +\infty$.
  In particular, (up to a subsequence) $u^n \to u$ pointwise a.e.\ in $\T^1$ and, by Fatou's lemma, we have that
  \[
  \|u\|_{\W^\alpha_\delta(\T^1)} \leq \liminf_{n \to +\infty} \|u^n\|_{\W^\alpha_\delta(\T^1)} \leq \lim_{n \to +\infty} C_2 \|u^n\|_{H^{\alpha}(\T^1)} = C_2 \|u\|_{H^{\alpha}(\T^1)} \, ,
  \]
  showing~\eqref{eq:W_equiv_H_1}.

  Let us prove~\eqref{eq:W_equiv_H_2}.
  Let $u \in \W^\alpha_\delta(\T^1)$ and let $\rho_\e$ be a family of mollifiers, \ie, $\rho_\e(x) = \frac{1}{\e} \rho\big(\frac{x}{\e}\big)$, where $\rho \in C^\infty_c(\R)$ with $\rho \geq 0$, $\supp \rho_\e \subset (-\frac{1}{2}, \frac{1}{2})$, and $\int_\R \rho(x) \d x = 1$.
  Extending $\rho_\e$ by 1-periodicity, we define $u^\e(x) = u * \rho_\e(x) = \int_{\T^1} u(y) \rho_\e(x-y) \d y \in C^\infty(\T^1)$.
  On the one hand, by Jensen's inequality, we have that
  \[ \label{eq:W_conv}
  \begin{split}
    [u^\e]_{\W^\alpha_\delta(\T^1)}^2 & = \frac{1}{2} \int_{\T^1} \int_{\R} \chi_\delta(y) \frac{|u^\e(x+y) - u^\e(x)|^2}{|y|^{1+2\alpha}} \, \d y \, \d x \\
    & = \frac{1}{2} \int_{\T^1} \int_{\R} \frac{\chi_\delta(y)}{|y|^{1+2\alpha}} \Big| \int_{\T^1} (u(z+y) - u(z)) \rho_\e(x-z) \d z \Big|^2 \, \d y \, \d x \\
    & \leq \frac{1}{2} \int_{\T^1} \int_{\R} \frac{\chi_\delta(y)}{|y|^{1+2\alpha}} \int_{\T^1} |u(z+y) - u(z)|^2 \rho_\e(x-z) \d z \, \d y \, \d x \\
    & = \frac{1}{2} \int_{\T^1} \int_{\R} \frac{\chi_\delta(y)}{|y|^{1+2\alpha}} \int_{\T^1} |u(z+y) - u(z)|^2 \rho_\e(x-z) \d x \, \d y \, \d z \\
    & = \frac{1}{2} \int_{\T^1} \int_{\R} \chi_\delta(y) \frac{|u(z+y) - u(z)|^2}{|y|^{1+2\alpha}} \, \d y \, \d z = [u]_{\W^\alpha_\delta(\T^1)}^2 \, .
  \end{split}
  \]
  On the other hand, we fix $n \in \N$ and, since $\hat{\rho}_{\e,k} \to 1$ as $\e \to 0$ for every $|k| \leq n$, we have that
  \[
  |\hat \rho_{\e,k}| \geq \frac{1}{2} \, , \quad \text{for every } |k| \leq n \text{ and } \e > 0 \text{ small enough} \, ,
  \]
  implying that $u^n = \sum_{|k| \leq n} \hat{u}_k e^{i 2 \pi k x} \in C^\infty(\T^1)$ satisfies
  \[
  \begin{split}
    \| u^n \|_{H^\alpha(\T^1)}^2 & = \sum_{|k| \leq n} (1 + |k|^2)^\alpha |\hat{u}_k|^2 \leq 4 \sum_{k \in \Z} (1 + |k|^2)^\alpha |\hat{\rho}_{\e,k}|^2 |\hat{u}_k|^2 \\
    & \leq C \sum_{k \in \Z} (1 + \omega^2(k)) |\hat{\rho}_{\e,k}|^2 |\hat{u}_k|^2 = C \|u^\e\|_{\W^\alpha_\delta(\T^1)}^2 \, ,
  \end{split}
  \]
  for $\e > 0$ small enough and some constant $C > 0$ (depending on $\alpha$ and $\delta$).
  Combining~\eqref{eq:W_conv} with the previous estimate and Young's inequality $\|u^\varepsilon\|_{L^2(\T^1)} \leq \|u\|_{L^2(\T^1)}$, we deduce that
  \[
  \| u^n \|_{H^\alpha(\T^1)} \leq C \|u\|_{\W^\alpha_\delta(\T^1)} \, , \quad \text{for every } n \in \N \, ,
  \]
  which proves~\eqref{eq:W_equiv_H_2} by letting $n \to +\infty$.

\end{proof}

Thanks to the previous proposition, we will identify the energy space $\W^\alpha_\delta(\T^1)$ with the fractional Sobolev space $H^{\alpha}(\T^1)$ in the rest of the paper.

\subsection{Well-posedness} We show here existence and uniqueness of solutions to the Cauchy problem~\eqref{eq:peridynamics} in the fractional Sobolev space $H^\alpha(\T^1)$.

Solutions are intended in the sense of the following definition.

\begin{definition} \label{def:distributional-solution}
    Let $u_0, v_0 \in L^1_{\mathrm{loc}}(\T^1)$.
    We say that $u \in L^1_{\mathrm{loc}}([0,+\infty) \times \T^1)$ is a \emph{distributional solution} to the Cauchy problem~\eqref{eq:peridynamics} if
    \[
        \begin{split}
        & \int_0^{+\infty} \int_{\T^1} u(t,x) \de_{tt} \varphi(t,x) \, \d x \, \d t - \int_0^{+\infty} \int_{\T^1} u(t,x) K[\varphi(t,\cdot)](x) \, \d x \, \d t\\
        & \quad  = - \int_{\T^1} u_0(x) \de_t \varphi(0,x) \, \d x + \int_{\T^1} v_0(x) \varphi(0,x) \, \d x
        \end{split}
    \]
for every $\varphi \in C^\infty_c([0,+\infty) \times \T^1)$.
\end{definition}

We start with existence, which is obtained by exhibiting an explicit solution formula via Fourier series, exploiting the representation of the peridynamic operator as a Fourier multiplier operator given in Lemma~\ref{lem:K_fourier}.
For well-posedness in a more general setting, we refer to~\cite{CocDipMadVal18}.

\begin{theorem} \label{thm:existence}
Let $\alpha \in (0,1)$ and let $\delta > 0$.
Let
\[
u_0(x) = \sum_{k \in \Z} \hat{u}_{0,k} e^{i 2 \pi k x} \in H^\alpha(\T^1)\, , \quad v_0(x) = \sum_{k \in \Z} \hat{v}_{0,k} e^{i 2 \pi k x} \in L^2(\T^1)
\]
be given initial data.
Let
\[ \label{eq:solution_formula}
u(t,x) = \sum_{k \in \Z} \hat u_k(t) e^{i 2 \pi k x} \, , \quad \text{where } \hat u_k(t) = \hat{u}_{0,k} \cos(\omega(k) t) + \hat{v}_{0,k}\frac{\sin(\omega(k) t)}{\omega(k)}  \, , \quad \text{for every } k \in \Z \, ,
\]
with the convention that $\frac{\sin(\omega(k) t)}{\omega(k)} = t$ for $k = 0$.

Then $u \in C([0,+\infty);H^\alpha(\T^1)) \cap C^1([0,+\infty);L^2(\T^1))$ and it is a distributional solution to the Cauchy problem~\eqref{eq:peridynamics}.
\end{theorem}
\begin{proof}
Let us show that $u \in C([0,+\infty);H^\alpha(\T^1)) \cap C^1([0,+\infty);L^2(\T^1))$.
First of all, by~\eqref{eq:solution_formula} we have that
\[
\begin{split}
  \|u(t,\cdot)\|_{H^\alpha(\T^1)} & \leq \Big( \sum_{k \in \Z} (1 + |k|^2)^\alpha |\hat u_{0,k}|^2 \Big)^\frac{1}{2} + \Big( \sum_{k \in \Z} (1 + |k|^2)^\alpha |\hat v_{0,k}|^2 \frac{\sin^2(\omega(k) t)}{\omega^2(k)} \Big)^\frac{1}{2} \, .
\end{split}
\]
The first term in the right-hand side is finite since $u_0 \in H^\alpha(\T^1)$.
As for the second term, by Lemma~\ref{lem:K_fourier} we have that
\[ \label{eq:velocity_bound}
\sum_{k \in \Z} (1 + |k|^2)^\alpha |\hat v_{0,k}|^2 \frac{\sin^2(\omega(k) t)}{\omega^2(k)} \leq |\hat v_{0,0}|^2 t^2 + C \sum_{|k| \geq 1} |\hat v_{0,k}|^2  \, ,
\]
which is finite since $v_0 \in L^2(\T^1)$.
To show that $u \in C([0,+\infty);H^\alpha(\T^1))$, we estimate
\[
\begin{split}
  & \|u(t,\cdot) - u(s,\cdot)\|_{H^\alpha(\T^1)} = \Big( \sum_{k \in \Z} (1 + |k|^2)^\alpha |\hat u_k(t) - \hat u_k(s)|^2 \Big)^\frac{1}{2} \\
  & \quad \leq \Big( \sum_{k \in \Z} (1 + |k|^2)^\alpha |\hat u_{0,k}|^2 |\cos(\omega(k) t) - \cos(\omega(k) s)|^2 \Big)^\frac{1}{2} \\
  & \quad \quad + \Big( \sum_{k \in \Z} (1 + |k|^2)^\alpha |\hat v_{0,k}|^2 \frac{|\sin(\omega(k) t) - \sin(\omega(k) s)|^2}{\omega^2(k)} \Big)^\frac{1}{2} \, ,
\end{split}
\]
which goes to $0$ as $t \to s$ by the Dominated Convergence Theorem, since $u_0 \in H^\alpha(\T^1)$ and $v_0 \in L^2(\T^1)$ and by Lemma~\ref{lem:K_fourier}.

Finally, to show that $u \in C^1([0,+\infty);L^2(\T^1))$, we let
\[ \label{eq:velocity_formula}
v(t,x) := \sum_{k \in \Z} \hat u_k'(t) e^{i 2 \pi k x} = \sum_{k \in \Z} \Big( - \omega(k) \hat{u}_{0,k} \sin(\omega(k) t) + \hat{v}_{0,k} \cos(\omega(k) t) \Big) e^{i 2 \pi k x} \, ,
\]
and we compute
\[
\begin{split}
  & \Big\|\frac{u(t,\cdot) - u(s,\cdot)}{t - s} - v(s) \Big\|_{L^2(\T^1)} = \Big( \sum_{k \in \Z} \Big| \frac{\hat u_k(t) - \hat u_k(s)}{t - s} - \hat v_k(s) \Big|^2 \Big)^\frac{1}{2} \\
  & \quad \leq \Big( \sum_{k \in \Z} \Big| \frac{\cos(\omega(k) t) - \cos(\omega(k) s)}{t - s} + \omega(k) \sin(\omega(k) s) \Big|^2 |\hat u_{0,k}|^2 \Big)^\frac{1}{2} \\
  & \quad \quad + \Big( \sum_{k \in \Z} \Big| \frac{\sin(\omega(k) t) - \sin(\omega(k) s)}{\omega(k)(t - s)} - \cos(\omega(k) s) \Big|^2 |\hat v_{0,k}|^2 \Big)^\frac{1}{2} \, ,
\end{split}
\]
which goes to $0$ as $t \to s$ by the Dominated Convergence Theorem since $u_0 \in H^\alpha(\T^1)$ and $v_0 \in L^2(\T^1)$ and by Lemma~\ref{lem:K_fourier}.
Continuity of $t \mapsto v(t) \in L^2(\T^1)$ is a consequence of the Dominated Convergence Theorem applied to~\eqref{eq:velocity_formula}.

Let us now prove that $u$ is a distributional solution, arguing by approximation with smooth solutions.
Let $n \in \N$ and let $u^n(t,x) = \sum_{|k| \leq n} \hat u_k(t) e^{i 2 \pi k x}$ with $\hat u_k(t)$ as in~\eqref{eq:solution_formula}.
Then $u^n$ is a smooth solution to the Cauchy problem~\eqref{eq:peridynamics} with initial data $u^n(0,x) = u_0^n := \sum_{|k| \leq n} \hat{u}_{0,k} e^{i 2 \pi k x}$ and $\de_t u^n(0,x) = v_0^n := \sum_{|k| \leq n} \hat{v}_{0,k} e^{i 2 \pi k x}$.
Indeed, the Fourier coefficients of $u^n$ (coinciding with those of $u$ for $|k| \leq n$) satisfy the ODEs
\[
\de_{tt} \hat u_k(t) + \omega^2(k) \hat u_k(t) = 0 \, , \quad \text{for } |k| \leq n \, .
\]
By Lemma~\ref{lem:K_fourier}, we have that
\[
\de_{tt} \hat u^n_k(t) - \hat{K[u^n(t,\cdot)]}_k = 0 \, , \quad \text{for } |k| \leq n \, ,
\]
which, after multiplying by $e^{i 2 \pi k x}$ and summing over $|k| \leq n$, yields that
\[
\de_{tt} u^n(t,x) - K[u^n(t,\cdot)](x) = 0 \, .
\]
The initial conditions $u^n(0,x) = u_0^n$ and $\de_t u^n(0,x) = v_0^n$ are satisfied by construction.
Let us now establish the convergence of $u^n$ to $u$.
Arguing as in~\eqref{eq:velocity_bound}, we have that
\[
\begin{split}
  & \|u^n(t,\cdot) - u(t,\cdot)\|_{H^\alpha(\T^1)} = \Big( \sum_{|k| > n} (1 + |k|^2)^\alpha |\hat u_k(t)|^2 \Big)^\frac{1}{2} \\
  & \quad \leq \Big( \sum_{|k| > n} (1 + |k|^2)^\alpha |\hat u_{0,k}|^2 \Big)^\frac{1}{2} + C \Big( \sum_{|k| > n} |\hat v_{0,k}|^2 \Big)^\frac{1}{2} \, ,
\end{split}
\]
which goes to $0$ as $n \to \infty$ uniformly in $t$, since $u_0 \in H^\alpha(\T^1)$ and $v_0 \in L^2(\T^1)$.
Similarly, by Lemma~\ref{lem:K_fourier} we have that
\[
\begin{split}
  & \|\de_t u^n(t,\cdot) - \de_t u(t,\cdot)\|_{L^2(\T^1)} = \Big( \sum_{|k| > n} |\hat u_k'(t)|^2 \Big)^\frac{1}{2} \\
  & \quad \leq \Big( \sum_{|k| > n} \omega^2(k)|\hat u_{0,k}|^2 \Big)^\frac{1}{2} + \Big( \sum_{|k| > n} |\hat v_{0,k}|^2 \Big)^\frac{1}{2} \leq C \Big( \sum_{|k| > n} (1 + |k|^2)^\alpha |\hat u_{0,k}|^2 \Big)^\frac{1}{2} + \Big( \sum_{|k| > n} |\hat v_{0,k}|^2 \Big)^\frac{1}{2} \, ,
\end{split}
\]
which goes to $0$ as $n \to \infty$ uniformly in $t$, since $u_0 \in H^\alpha(\T^1)$ and $v_0 \in L^2(\T^1)$.
These convergences are sufficient to pass to the limit in the weak formulation of the equation satisfied by $u^n$, which reads as
\[
\begin{split}
& \int_0^{+\infty} \int_{\T^1} u^n(t,x) \de_{tt} \varphi(t,x) \, \d x \, \d t - \int_0^{+\infty} \int_{\T^1} u^n(t,x) K[\varphi(t,\cdot)](x) \, \d x \, \d t\\
& \quad  = - \int_{\T^1} u_0^n(x) \de_t \varphi(0,x) \, \d x + \int_{\T^1} v_0^n(x) \varphi(0,x) \, \d x \, ,
\end{split}
\]
for every $\varphi \in C^\infty_c([0,+\infty) \times \T^1)$.
\end{proof}

In the next theorem, we show uniqueness of distributional solutions to the Cauchy problem~\eqref{eq:peridynamics} in the class $C([0,+\infty);H^\alpha(\T^1)) \cap C^1([0,+\infty);L^2(\T^1))$.

\begin{theorem} \label{thm:uniqueness}
Let $\alpha \in (0,1)$ and let $\delta > 0$.
Let $u \in C([0,+\infty);H^\alpha(\T^1)) \cap C^1([0,+\infty);L^2(\T^1))$ be a distributional solution to the Cauchy problem~\eqref{eq:peridynamics} with initial data $u_0 \in H^\alpha(\T^1)$ and $v_0 \in L^2(\T^1)$.
Then $u$ is the solution given by the formula~\eqref{eq:solution_formula}.
\end{theorem}
\begin{proof}
First of all, note that
\[
\begin{split}
  K[e^{-i 2 \pi k \cdot}](x) & = \frac{1}{2} \int_{\R} \chi_\delta(y) \frac{e^{-i 2 \pi k (x+y)} + e^{-i 2 \pi k (x-y)} - 2 e^{-i 2 \pi k x}}{|y|^{1+2\alpha}} \, \d y \\
  &  = e^{-i 2 \pi k x} \frac{1}{2} \int_{\R} \chi_\delta(y) \frac{e^{-i 2 \pi k y} + e^{i 2 \pi k y} - 2}{|y|^{1+2\alpha}} \, \d y = - e^{-i 2 \pi k x} \omega^2(k) \, .
\end{split}
\]
Hence, using $\varphi(t,x) := \psi(t) e^{-i 2 \pi k x}$ with $\psi \in C^\infty_c([0,+\infty))$ in the definition of distributional solution, we obtain that
\[
\begin{split}
& \int_0^{+\infty} \int_{\T^1} u(t,x) \de_{tt} \varphi(t,x) \, \d x \, \d t - \int_0^{+\infty} \int_{\T^1} u(t,x) K[\varphi(t,\cdot)](x) \, \d x \, \d t\\
& \quad  = - \int_{\T^1} u_0(x) \de_t \varphi(0,x) \, \d x + \int_{\T^1} v_0(x) \varphi(0,x) \, \d x \\
& \implies  \int_0^{+\infty} \de_{tt}\psi(t) \hat u_k(t) \, \d t +  \int_0^{+\infty} \psi(t) \omega^2(k) \hat u_k(t) \, \d t = - \de_t \psi(0) \hat u_{0,k} + \psi(0) \hat v_{0,k} \, ,
\end{split}
\]
that is, $\hat u_k(t)$ is a distributional solution to the ODE
\[
\de_{tt} \hat u_k(t) + \omega^2(k) \hat u_k(t) = 0 \, ,
\]
with initial conditions $\hat u_{0,k}$ and $\hat v_{0,k}$.
Since $u \in C^1([0,+\infty);L^2(\T^1))$, we have $\hat u_k(t) \in C^1([0,+\infty))$ and thus, by a standard bootstrap argument, $\hat u_k$ is a classical solution to the above ODE.
Then the conclusion follows by uniqueness of solutions to the above ODE, which yields that $\hat u_k(t)$ is given by~\eqref{eq:solution_formula}.
\end{proof}

We conclude this section by observing energy conservation of the solutions.

\begin{theorem} \label{thm:energy_conservation}
Let $\alpha \in (0,1)$ and let $\delta > 0$.
Let $u \in C([0,+\infty);H^\alpha(\T^1)) \cap C^1([0,+\infty);L^2(\T^1))$ be the unique distributional solution to the Cauchy problem~\eqref{eq:peridynamics} with initial data $u_0 \in H^\alpha(\T^1)$ and $v_0 \in L^2(\T^1)$ provided by Theorems~\ref{thm:existence} and~\ref{thm:uniqueness}.
Then the energy
\[
\frac{1}{2} \|\de_t u(t,\cdot)\|_{L^2(\T^1)}^2 + \frac{1}{2} [u(t,\cdot)]_{\W^\alpha_\delta(\T^1)}^2
\]
is constant for $t \geq 0$.
\end{theorem}
\begin{proof}
    This can be proved directly via the representation of the solution via Fourier series.
    Indeed, by Theorem~\ref{thm:existence} (see also~\eqref{eq:velocity_formula} within its proof), we have that
    \[
    \de_t u(t,x) = \sum_{k \in \Z} \Big( - \omega(k) \hat{u}_{0,k} \sin(\omega(k) t) + \hat{v}_{0,k} \cos(\omega(k) t) \Big) e^{i 2 \pi k x} \, .
    \]
    Hence, by Parseval's identity and by Lemma~\ref{lem:K_fourier}, we compute
    \[
    \begin{split}
        & \|\de_t u(t,\cdot)\|_{L^2(\T^1)}^2 + [u(t,\cdot)]_{\W^\alpha_\delta(\T^1)}^2 = \sum_{k \in \Z} \Big( |\hat u_k'(t)|^2 + \omega^2(k) |\hat u_k(t)|^2 \Big) \\
        & \quad = \sum_{k \in \Z} \Big( \Big| - \omega(k) \hat{u}_{0,k} \sin(\omega(k) t) + \hat{v}_{0,k} \cos(\omega(k) t) \Big|^2 +\omega^2(k) \Big| \hat{u}_{0,k} \cos(\omega(k) t) + \hat{v}_{0,k}\frac{\sin(\omega(k) t)}{\omega(k)} \Big|^2 \Big) \\
        & \quad = \sum_{k \in \Z} \Big( |\hat{v}_{0,k}|^2 + \omega^2(k) |\hat{u}_{0,k}|^2 \Big) = \|\de_t u(0,\cdot)\|_{L^2(\T^1)}^2 + [u(0,\cdot)]_{\W^\alpha_\delta(\T^1)}^2 \, ,
    \end{split}
    \]
    which shows that the energy is constantly equal to its initial value.
\end{proof}

\section{An example of dissolution of a jump discontinuity} \label{sec:dissolution}

This section is at the core of the main result of the paper.
Here, we provide an explicit example of an initial datum $u_0$ (with initial velocity $v_0 = 0$) for which a jump discontinuity dissolves instantaneously in the peridynamic evolution.
More precisely, we construct $u_0 \in H^\alpha(\T^1)$ with the following property: $u_0$ has one jump discontinuity, while the solution $u(t,\cdot)$ to the Cauchy problem~\eqref{eq:peridynamics} is continuous for $t > 0$.

The main results of this section hold true for $\alpha \in (0,\frac{1}{2})$.
The initial datum that we choose has a jump discontinuity and belongs to $H^\alpha(\T^1)$ only for $\alpha < \frac{1}{2}$.

\subsection{The initial data} We start by defining the promised initial data.
We let
\[ \label{eq:initial_data}
u_0(x) := \frac{1}{2} - x \, , \quad  v_0(x) := 0 \, , \quad \text{for } x \in [0,1) \, ,
\]
extended by 1-periodicity to the whole real line.
Note that $u_0$ has a jump discontinuity at $x = 0$ (identified with $x = 1$ in $\T^1$).

\begin{remark} \label{rem:u0}
  The Fourier coefficients of $u_0$ are given by
  \[
  \hat u_{0,k} = \int_{\T^1} u_0(x) e^{-i 2 \pi k x} \, \d x = \int_{0}^{1} \Big( \frac{1}{2} - x \Big) e^{-i 2 \pi k x} \, \d x \, .
  \]
  For $k = 0$, we have that $\hat u_{0,0} = 0$.
  For $k \neq 0$, integrating by parts we obtain that
  \[
  \hat u_{0,k} = - \int_0^1 x e^{-i 2 \pi k x} \, \d x = \Big[\frac{x}{i 2 \pi k} e^{-i 2 \pi k x} \Big]_0^1 - \frac{1}{i 2 \pi k} \int_0^1 e^{-i 2 \pi k x} \, \d x = \frac{1}{i 2 \pi k} \, .
  \]

  In particular, we have that $u_0 \in H^\alpha(\T^1)$ for every $\alpha \in (0,\frac{1}{2})$.
  Indeed,
  \[
  \|u_0\|_{H^\alpha(\T^1)}^2 = \sum_{k \in \Z} (1 + |k|^2)^\alpha |\hat u_{0,k}|^2 = \sum_{k \in \Z \setminus \{0\}} (1 + |k|^2)^\alpha \frac{1}{4\pi^2 k^2} \leq C \sum_{k = 1}^{+\infty} \frac{1}{k^{2 - 2\alpha}} < +\infty \, ,
  \]
  since $2 - 2\alpha > 1$.
\end{remark}

\subsection{Main result about the dissolution of the jump discontinuity}
The main result in this section is the following.

\begin{theorem} \label{thm:dissolution}
  Let $\alpha \in (0,\frac{1}{2})$.
  Let $u_0 \in H^\alpha(\T^1)$ and $v_0 \in L^2(\T^1)$ be as in~\eqref{eq:initial_data}.
  Let $u \in C([0,+\infty);H^\alpha(\T^1)) \cap C^1([0,+\infty);L^2(\T^1))$ be the unique solution to the Cauchy problem~\eqref{eq:peridynamics} with initial data $u_0$ and $v_0$ provided by Theorem~\ref{thm:existence}.
  Then $u(t,\cdot)$ is a continuous function for every $t > 0$.
  Moreover, it satisfies the bound 
 \begin{equation} \label{eq:L_infty_bound_solution}
     \| u(t,\cdot) \|_{L^\infty(\T^1)} \leq C t^{-\frac{1}{2}} \, , \quad \text{for every } t \in (0,1]\, .
 \end{equation} 
\end{theorem}

The proof of the previous result is postponed to Subsection~\ref{subsec:proof-dissolution}, after proving a series of technical preparatory results.

\subsection{Results on oscillatory series} \label{subsec:technical}
The following theorem is the key ingredient in the proof of Theorem~\ref{thm:dissolution}.

\begin{theorem} \label{thm:series_convergence}
  Let $\alpha \in (0,1)$.
  Let $t > 0$.
  Then the series
  \[
  \sum_{k \in \Z \setminus \{0\}} \frac{1}{k} e^{i 2\pi k x + i \omega(k) t}
  \]
  converges uniformly in $x \in \big[-\frac{1}{2},\frac{1}{2}\big]$ to a continuous function.
  Moreover,
  \begin{equation} \label{eq:series_L_infty_bound}
  \sup_{x \in [-\frac{1}{2},\frac{1}{2}]} \Big| \sum_{k \in \Z \setminus \{0\}} \frac{1}{k} e^{i 2\pi k x + i \omega(k) t} \Big| \leq C t^{-\frac{1}{2}} \, , \quad \text{for every } t \in (0,1] \, .
\end{equation}
\end{theorem}

The proof of Theorem~\ref{thm:series_convergence} is postponed, since it relies on a series of lemmata in the spirit of van der Corput's method for oscillatory integrals.
Some of the lemmata are standard, while others are more specific to the peridynamic setting.
We provide all the details of the proof for the sake of completeness, referring to~\cite[Chapter 5]{Zyg59} for a classical discussion on the topic.

\vspace{1em}

In the next steps, it is convenient to introduce the notation
\[
\phi(\xi,x,t) := 2\pi \xi x + \omega(\xi) t \, , \quad \text{for } \xi \in \R \, , \ x \in \Big[-\frac{1}{2}, \frac{1}{2}\Big] \, , \ t > 0 \, .
\]
Note that the interval $[-\frac{1}{2}, \frac{1}{2}]$ is chosen as a representative of $\T^1$.

We also introduce the notation
\[
\bar \phi(\xi,x,t) := 2\pi \xi x + \gamma |\xi|^\alpha t \, , \quad \text{for } \xi \in \R \, , \ x \in \Big[-\frac{1}{2}, \frac{1}{2}\Big] \, , \ t > 0 \, .
\]

In the next lemmata, we study the oscillatory series
\[
\sum_{k=1}^{+\infty} \frac{1}{k} e^{i \phi(k,x,t)} \, , \quad \text{and} \quad \sum_{k=1}^{+\infty} \frac{1}{k} e^{i \bar \phi(k,x,t)} \, ,  \quad \text{for } t > 0 \, , \ x \in \Big[-\frac{1}{2}, \frac{1}{2}\Big] \, .
\]
Notice that the stated results hold true for the series over negative integers.

\vspace{1em}

The first step is to compare an oscillatory series involving $\bar \phi$ with the corresponding oscillatory integral, which is easier to treat.
The precise statement is the following.

\begin{lemma} \label{lem:series_integral_comparison}
  Let $\alpha \in (0,1)$.
  Then there exists $C > 0$ with the following property:
  For every $t > 0$ there exists $N_0 \in \N$, $N_0 \geq 1$ such that
  \[
  \Big| \sum_{k = N_0}^N e^{i \bar \phi(k,x,t)} - \int_{N_0}^N e^{i \bar \phi(\xi,x,t)} \, \d \xi \Big| \leq C  \, ,
  \]
  for every $N \geq N_0$ and every $x \in \big[-\frac{1}{2},\frac{1}{2}\big]$.
  Moreover, $N_0 \sim t^{\frac{1}{1-\alpha}}$ for $t$ large.
\end{lemma}
\begin{proof}
    We follow the proof as proposed in~\cite{Fos05}, providing all the details for the sake of completeness.
    Let
    \[ \label{eq:N0_choice}
    N_0 = \Big\lceil  \Big(\frac{2\alpha \gamma t}{\pi}\Big)^{\frac{1}{1-\alpha}} \Big\rceil \, ,
    \]
    where $\gamma$ is as in~\eqref{eq:omega_limit} and $\lceil \cdot \rceil$ denotes the ceiling function.
    The reason for this specific choice of $N_0$ will be clear in~\eqref{eq:derivative_phi_bound} below.
    For the moment, the reader can think of $N_0$ as a large enough integer.

    On the one hand, summation by parts yields
    \[ \label{eq:summation_by_parts}
    \begin{split}
        \sum_{k=N_0}^N e^{i \bar \phi(k,x,t)} & = \sum_{k=N_0}^{N} (k - (k-1)) e^{i \bar \phi(k,x,t)} = \sum_{k=N_0}^{N} k e^{i \bar \phi(k,x,t)} - \sum_{k=N_0-1}^{N-1} k e^{i \bar \phi(k+1,x,t)} \\
        & = N e^{i \bar \phi(N,x,t)} - (N_0 - 1) e^{i \bar \phi(N_0,x,t)} - \sum_{k=N_0}^{N-1} k (e^{i \bar \phi(k+1,x,t)} - e^{i \bar \phi(k,x,t)}) \\
        & = N e^{i \bar \phi(N,x,t)} - (N_0 - 1) e^{i \bar \phi(N_0,x,t)} - \sum_{k=N_0}^{N-1} k \int_{k}^{k+1} \de_\xi e^{i \bar \phi(\xi,x,t)} \, \d \xi \\
        & = N e^{i \bar \phi(N,x,t)} - (N_0 - 1) e^{i \bar \phi(N_0,x,t)} - \int_{N_0}^{N} \lfloor \xi \rfloor \de_\xi e^{i \bar \phi(\xi,x,t)} \, \d \xi \, ,
    \end{split}
    \]
    where $\lfloor \xi \rfloor$ denotes the integer part of $\xi$.
    On the other hand, integrating by parts we obtain
    \[ \label{eq:integration_by_parts}
    \begin{split}
        \int_{N_0}^N e^{i \bar \phi(\xi,x,t)} \, \d \xi & = N e^{i \bar \phi(N,x,t)} - N_0 e^{i \bar \phi(N_0,x,t)} - \int_{N_0}^N \xi \de_\xi e^{i \bar \phi(\xi,x,t)} \, \d \xi \, .
    \end{split}
    \]
    Putting together~\eqref{eq:summation_by_parts} and~\eqref{eq:integration_by_parts}, we get that
    \[ \label{eq:series_minus_integral}
    \sum_{k=N_0}^N e^{i \bar \phi(k,x,t)} - \int_{N_0}^N e^{i \bar \phi(\xi,x,t)} \, \d \xi = e^{i \bar \phi(N_0,x,t)} + \int_{N_0}^N (\xi - \lfloor \xi \rfloor) \de_\xi e^{i \bar \phi(\xi,x,t)} \, \d \xi \, .
    \]
    The term $e^{i \bar \phi(N_0,x,t)}$ is bounded by $1$.
    To estimate the integral, we observe that $\xi \mapsto \xi - \lfloor \xi \rfloor$ is periodic with period $1$ and can be written as a Fourier series.
    Its coefficients are given by
    \[
    \int_0^1 (\xi - \lfloor \xi \rfloor) e^{-i 2 \pi k \xi} \, \d \xi = \int_0^1 \xi e^{-i 2 \pi k \xi} \, \d \xi =
    \begin{cases} \displaystyle \frac{1}{2}\,, & \text{if } k = 0 \, , \\
      \displaystyle \frac{i}{2\pi k}\,, & \text{if } k \neq 0 \, ,
    \end{cases}
    \]
    hence,
    \[ \label{eq:fourier_series_of_fractional_part}
    \begin{split}
      & \int_{N_0}^N (\xi - \lfloor \xi \rfloor) \de_\xi e^{i \bar \phi(\xi,x,t)} \, \d \xi \\
      & \quad = \int_{N_0}^N \Big( \frac{1}{2} + \sum_{k \in \Z \setminus \{0\}} \frac{i}{2\pi k} e^{i 2 \pi k \xi} \Big) \de_\xi e^{i \bar \phi(\xi,x,t)} \, \d \xi \\
      & \quad = \frac{1}{2} e^{i \bar \phi(N,x,t)} - \frac{1}{2} e^{i \bar \phi(N_0,x,t)} + \sum_{k \in \Z \setminus \{0\}} \frac{i}{2\pi k} \int_{N_0}^N e^{i 2 \pi k \xi} \de_\xi e^{i \bar \phi(\xi,x,t)} \, \d \xi \, ,
    \end{split}
    \]
    where in the last equality we exchanged the order of the integral and the sum thanks to the convergence in $L^2$ of the sawtooth Fourier partial sums.  

    Let us estimate the integrals $\int_{N_0}^N e^{i 2 \pi k \xi} \de_\xi e^{i \bar \phi(\xi,x,t)} \, \d \xi$ for every $k \in \Z \setminus \{0\}$.
    We have that
    \[ \label{eq:integral_to_estimate}
    \begin{split}
      \int_{N_0}^N e^{i 2 \pi k \xi} \de_\xi e^{i \bar \phi(\xi,x,t)} \, \d \xi & = \int_{N_0}^N e^{i (2\pi k \xi + \bar \phi(\xi,x,t))} i \de_\xi \bar \phi(\xi,x,t) \, \d \xi \\
      & = \int_{N_0}^N \frac{\de_\xi \bar \phi(\xi,x,t)}{ 2\pi k + \de_\xi \bar \phi(\xi,x,t)} \de_\xi e^{i (2\pi k \xi + \bar \phi(\xi,x,t))} \, \d \xi \\
      & = \int_{N_0}^N \psi(\xi,k,x,t) \de_\xi e^{i (2\pi k \xi + \bar \phi(\xi,x,t))} \, \d \xi \, ,
    \end{split}
    \]
    where we set $\psi(\xi,k,x,t) := \frac{\de_\xi \bar \phi(\xi,x,t)}{ 2\pi k + \de_\xi \bar \phi(\xi,x,t)}$.
    First of all, we observe that the denominator in $\psi$ does not vanish, since
    \[ \label{eq:derivative_phi_bound}
    |\de_\xi \bar \phi(\xi,x,t)| = |2\pi x + \alpha \gamma \xi^{-1+\alpha} t| \leq \pi + \alpha \gamma t N_0^{-1+\alpha} \leq \frac{3}{2} \pi < 2\pi \leq 2 \pi |k| \, ,
    \]
    for $N_0$ chosen as in~\eqref{eq:N0_choice} (which satisfies $\alpha \gamma t N_0^{-1+\alpha} < \pi$).
    From this inequality, it also follows that
    \[ \label{eq:psi_bound}
    |\psi(\xi,k,x,t)| \leq \frac{\frac{3}{2}\pi}{2\pi |k|- \frac{3}{2}\pi} \leq \frac{3}{|k|}  \, .
    \]
    Moreover, we have that
    \[ \label{eq:psi_derivative}
    \begin{split}
      \de_\xi \psi(\xi,k,x,t) & =  \frac{\de_{\xi \xi} \bar \phi(\xi,x,t)}{ 2\pi k + \de_\xi \bar \phi(\xi,x,t)} - \frac{\de_\xi \bar \phi(\xi,x,t)\de_{\xi \xi} \bar \phi(\xi,x,t)}{( 2\pi k + \de_\xi \bar \phi(\xi,x,t))^2} = \frac{2\pi k \de_{\xi \xi} \bar \phi(\xi,x,t)}{( 2\pi k + \de_\xi \bar \phi(\xi,x,t))^2} \\
      & = \frac{2\pi k \alpha(\alpha - 1) \gamma \xi^{-2+\alpha} t}{( 2\pi k + \de_\xi \bar \phi(\xi,x,t))^2} \, ,
    \end{split}
    \]
    which is negative for $k > 0$ and positive for $k < 0$, yielding monotonicity of $\psi(\xi,k,x,t)$ for every fixed $k \in \Z \setminus \{0\}$.
   We now integrate by parts in~\eqref{eq:integral_to_estimate} to obtain
    \[
    \begin{split}
     & \int_{N_0}^N \psi(\xi,k,x,t) \de_\xi e^{i (2\pi k \xi + \bar \phi(\xi,x,t))} \, \d \xi \\
     & \quad = \psi(N,k,x,t) e^{i (2\pi k N + \bar \phi(N,x,t))} - \psi(N_0,k,x,t) e^{i (2\pi k N_0 + \bar \phi(N_0,x,t))} \\
     & \quad \quad - \int_{N_0}^N \de_\xi \psi(\xi,k,x,t) e^{i (2\pi k \xi + \bar \phi(\xi,x,t))} \, \d \xi \, ,
    \end{split}
    \]
    which, by~\eqref{eq:integral_to_estimate}, \eqref{eq:psi_bound}, and~\eqref{eq:psi_derivative}, yields
    \[
    \begin{split}
      \Big| \int_{N_0}^N e^{i 2 \pi k \xi} \de_\xi e^{i \bar \phi(\xi,x,t)} \, \d \xi \Big| & \leq |\psi(N,k,x,t)| + |\psi(N_0,k,x,t)| + \int_{N_0}^N |\de_\xi \psi(\xi,k,x,t)| \, \d \xi \\
      &\leq \frac{C}{|k|} + \Big| \int_{N_0}^N \de_\xi \psi(\xi,k,x,t) \, \d \xi \Big|  = \frac{C}{|k|} + | \psi(N,k,x,t) - \psi(N_0,k,x,t)| \\
      & \leq \frac{C}{|k|} \, .
    \end{split}
    \]

    Inserting this estimate into~\eqref{eq:fourier_series_of_fractional_part} and by~\eqref{eq:series_minus_integral}, we obtain
    \[
    \Big| \sum_{k=N_0}^N e^{i \bar \phi(k,x,t)} - \int_{N_0}^N e^{i \bar \phi(\xi,x,t)} \, \d \xi \Big| \leq C  + C \sum_{k \in \Z \setminus \{0\}} \frac{1}{k^2} \leq C \, ,
    \]
    concluding the proof of the lemma.
\end{proof}

Next, one shows that the oscillatory term (in the $\xi$ variable) in the integral
\[
\Big| \int_{N_0}^N e^{i \bar \phi(\xi,x,t)} \, \d \xi \Big|
\]
introduces cancellations that improve the trivial bound given by $N$.
The precise statement is the following.

\begin{lemma} \label{lem:integral_exp}
  Let $\alpha \in (0,1)$.
  Then there exists $C > 0$ such that
  \[
  \Big| \int_{N_0}^N e^{i \bar \phi(\xi,x,t)} \, \d \xi \Big| \leq C t^{-\frac{1}{2}} N^{1 - \frac{\alpha}{2}} \, ,
  \]
  for every $t > 0$, $N_0 \geq 1$, $N \geq N_0$, and $x \in \big[-\frac{1}{2},\frac{1}{2}\big]$.
\end{lemma}
\begin{proof}
    We start by observing that
    \[
    \de_\xi \bar \phi(\xi,x,t) = 2\pi x + \alpha \gamma \xi^{-1+\alpha} t \, , \quad \de_{\xi \xi} \bar \phi(\xi,x,t) = \alpha(\alpha - 1) \gamma \xi^{-2+\alpha} t \, .
    \]
    For every $\xi \in [N_0, N]$ we have that
    \[ \label{eq:phi_xi_xx}
    \xi^{-2+\alpha} \geq N^{-2+\alpha} \implies \de_{\xi \xi} \bar \phi(\xi,x,t) \leq \alpha(\alpha - 1) \gamma N^{-2+\alpha} t < 0 \, .
    \]
    The function $\xi \mapsto \de_\xi \bar \phi(\xi,x,t)$ is strictly decreasing in $[N_0,N]$.
    It may change sign at most once in $[N_0,N]$.
    (It does if $x \in \big[-\frac{1}{2},0\big)$ and $(-\frac{\alpha \gamma t}{2\pi x})^\frac{1}{1-\alpha} \in [N_0,N]$).
    If it vanishes in $[N_0,N]$, we let $\xi_0 \in [N_0,N]$ be the unique point where $\de_\xi \bar \phi(\xi_0,x,t) = 0$. Otherwise, we set $\xi_0 := N$ if $\de_\xi \bar \phi(\xi,x,t) > 0$ for $\xi \in [N_0,N]$ or $\xi_0 = N_0$ if $\de_\xi \bar \phi(\xi,x,t) < 0$ for $\xi \in [N_0,N]$.
    We split the integral into two parts:
    \[ \label{eq:integral_split}
    \int_{N_0}^{N} e^{i \bar \phi(\xi,x,t)} \, \d \xi = \int_{N_0}^{\xi_0} e^{i \bar \phi(\xi,x,t)} \, \d \xi + \int_{\xi_0}^{N} e^{i \bar \phi(\xi,x,t)} \, \d \xi \, ,
    \]
    (the integrals over possibly empty intervals requiring no estimate).
    We discuss in detail the estimate of the first integral, showing that
    \[ \label{eq:integral_exp_first}
    \Big| \int_{N_0}^{\xi_0} e^{i \bar \phi(\xi,x,t)} \, \d \xi \Big| \leq C t^{-\frac{1}{2}} N^{1 - \frac{\alpha}{2}} \, .
    \]
    We fix $\zeta \in (N_0, \xi_0)$ (to be chosen later) and we study separately the two parts in the split integral
    \[
    \int_{N_0}^{\xi_0} e^{i \bar \phi(\xi,x,t)} \, \d \xi = \int_{N_0}^{\zeta} e^{i \bar \phi(\xi,x,t)} \, \d \xi + \int_{\zeta}^{\xi_0} e^{i \bar \phi(\xi,x,t)} \, \d \xi \, .
    \]
    Let us estimate the integral on $(N_0, \zeta)$.
    From~\eqref{eq:phi_xi_xx}, we get that
    \[  \label{eq:phi_xi_bound}
    \de_\xi \bar \phi(\xi,x,t) \geq \de_\xi \bar \phi(\zeta,x,t) \geq \alpha(1-\alpha) \gamma N^{-2+\alpha} t (\xi_0 - \zeta) =: \lambda > 0 \quad \text{for every } \xi \in (N_0, \zeta) \, .
    \]
    Then we write
    \[
    \de_\xi e^{i \bar \phi(\xi,x,t)} = i \de_\xi \bar \phi(\xi,x,t) e^{i \bar \phi(\xi,x,t)} \, ,
    \]
    and we integrate by parts to obtain
    \[
    \begin{split}
        \int_{N_0}^{\zeta} e^{i \bar \phi(\xi,x,t)} \, \d \xi & = \int_{N_0}^{\zeta} \frac{1}{i \de_\xi \bar \phi(\xi,x,t)} \de_\xi e^{i \bar \phi(\xi,x,t)} \, \d \xi  \\
        & = \Big[ \frac{e^{i \bar \phi(\xi,x,t)}}{i \de_\xi \bar \phi(\xi,x,t)} \Big]_{N_0}^{\zeta} - \int_{N_0}^{\zeta} e^{i \bar \phi(\xi,x,t)} \de_\xi \Big( \frac{1}{i \de_\xi \bar \phi(\xi,x,t)} \Big) \, \d \xi \, .
    \end{split}
    \]
    On the one hand, by~\eqref{eq:phi_xi_bound} and since $\xi \mapsto \de_\xi \bar \phi(\xi,x,t)$ is monotone, we have that
    \[
    \begin{split}
        \Big| \int_{N_0}^{\zeta} e^{i \bar \phi(\xi,x,t)} \, \d \xi \Big| & \leq \frac{1}{|\de_\xi \bar \phi(\zeta,x,t)|} + \frac{1}{|\de_\xi \bar \phi(N_0,x,t)|} + \int_{N_0}^{\zeta} \Big| \de_\xi \Big( \frac{1}{\de_\xi \bar \phi(\xi,x,t)} \Big) \Big| \, \d \xi \\
        & \leq \frac{2}{\lambda} + \Big| \int_{N_0}^{\zeta} \de_\xi \Big( \frac{1}{\de_\xi \bar \phi(\xi,x,t)} \Big) \, \d \xi \Big| = \frac{2}{\lambda} + \Big| \frac{1}{\de_\xi \bar \phi(\zeta,x,t)} - \frac{1}{\de_\xi \bar \phi(N_0,x,t)} \Big| \\
        & \leq \frac{4}{\lambda} = C t^{-1} N^{2-\alpha}  \frac{1}{\xi_0 - \zeta} \, .
    \end{split}
    \]
    On the other hand, we estimate the integral on $(\zeta, \xi_0)$ by the trivial bound
    \[
    \Big| \int_{\zeta}^{\xi_0} e^{i \bar \phi(\xi,x,t)} \, \d \xi \Big| \leq \xi_0 - \zeta \, .
    \]
    Putting together the above estimates, we obtain that
    \[
    \Big| \int_{N_0}^{\xi_0} e^{i \bar \phi(\xi,x,t)} \, \d \xi \Big| \leq C t^{-1}  N^{2-\alpha}  \frac{1}{\xi_0 - \zeta} + \xi_0 - \zeta \, .
    \]
    We optimize in $\zeta$, obtaining the following.
    If $C^\frac{1}{2} t^{-\frac{1}{2}} N^{1-\frac{\alpha}{2}} < \xi_0 - N_0$, then the optimum is attained for $\xi_0 - \zeta = C^{\frac{1}{2}} t^{-\frac{1}{2}} N^{1-\frac{\alpha}{2}}$ and
    \[
    \Big| \int_{N_0}^{\xi_0} e^{i \bar \phi(\xi,x,t)} \, \d \xi \Big| \leq  C t^{-\frac{1}{2}} N^{1 - \frac{\alpha}{2}} \, ,
    \]
    implying~\eqref{eq:integral_exp_first}.
    If, instead, $C^\frac{1}{2} t^{-\frac{1}{2}} N^{1-\frac{\alpha}{2}} \geq \xi_0 - N_0$, then~\eqref{eq:integral_exp_first} follows from the trivial bound
    \[
    \Big| \int_{N_0}^{\xi_0} e^{i \bar \phi(\xi,x,t)} \, \d \xi \Big| \leq  \xi_0 - N_0 \leq C t^{-\frac{1}{2}} N^{1 - \frac{\alpha}{2}} \, .
    \]

    To prove that the integral on $(\xi_0, N)$ in~\eqref{eq:integral_split} satisfies
    \[ \label{eq:integral_exp_second}
    \Big| \int_{\xi_0}^{N} e^{i \bar \phi(\xi,x,t)} \, \d \xi \Big| \leq C t^{-\frac{1}{2}} N^{1 - \frac{\alpha}{2}} \, ,
    \]
    we argue in a symmetric way, fixing $\zeta \in (\xi_0,N)$ and observing that
    \[
    \de_\xi \bar \phi(\xi,x,t) \leq \de_\xi \bar \phi(\zeta,x,t) \leq \alpha(\alpha - 1) \gamma N^{-2+\alpha} t (\zeta - \xi_0) < 0 \, , \quad \text{for every } \xi \in (\zeta,N) \, .
    \]
    The rest of the proof is analogous and we omit the details.

    Finally, putting together~\eqref{eq:integral_split},~\eqref{eq:integral_exp_first} and~\eqref{eq:integral_exp_second}, we conclude the proof of the lemma.
\end{proof}

From the previous lemmata, we deduce the following estimate on the oscillatory series involving~$\bar \phi$.

\begin{lemma} \label{lem:series_exp}
  Let $\alpha \in (0,1)$.
  Then there exists a constant $C$ such that
  \[
  \Big| \sum_{k = 1}^N e^{i \bar \phi(k,x,t)} \Big| \leq C ( 1 + t^{\frac{1}{1-\alpha}} + t^{-\frac{1}{2}} N^{1 - \frac{\alpha}{2}} ) \, ,
  \]
  for every $t>0$, $N \geq 1$, and every $x \in \big[-\frac{1}{2},\frac{1}{2}\big]$.
\end{lemma}
\begin{proof}
  We let $N_0$ be as in Lemma~\ref{lem:series_integral_comparison}.
  If $N \leq N_0$, we estimate the series by the trivial bound
    \[
    \Big| \sum_{k = 1}^N e^{i \bar \phi(k,x,t)} \Big| \leq N \leq N_0 \leq C(1 + t^{\frac{1}{1-\alpha}}) \, .
    \]
    If $N > N_0$, by Lemma~\ref{lem:series_integral_comparison} and Lemma~\ref{lem:integral_exp}, we estimate
  \[
  \begin{split}
    \Big| \sum_{k = 1}^N e^{i \bar \phi(k,x,t)} \Big| & \leq N_0 + \Big| \sum_{k = N_0}^N e^{i \bar \phi(k,x,t)} - \int_{N_0}^N e^{i \bar \phi(\xi,x,t)} \, \d \xi \Big| + \Big| \int_{N_0}^N e^{i \bar \phi(\xi,x,t)} \, \d \xi \Big| \\
    & \leq C t^{\frac{1}{1-\alpha}} + C + C  t^{-\frac{1}{2}} N^{1 - \frac{\alpha}{2}} \, ,
    \end{split}
  \]
  concluding the proof of the lemma.
\end{proof}

We are finally in a position to prove the main result of this subsection, Theorem~\ref{thm:series_convergence}.

\begin{proof}[Proof of Theorem~\ref{thm:series_convergence}]
  Let us set
  \[
  s(k,x,t) := \sum_{h=1}^k e^{i \bar \phi(h,x,t)} \, , \quad \rho(\xi,x,t) := \phi(\xi,x,t) - \bar \phi(\xi,x,t) \, ,
  \]
  for $k \in \N$, $k \geq 1$, $\xi > 0$, $x \in [-\frac{1}{2},\frac{1}{2}]$, and $t > 0$.
  Moreover, we set $s(0,x,t) := 0$.
  We apply a summation by parts to obtain, for every $N \geq N_0$,
  \[ \label{eq:summation_by_parts_replacement}
  \begin{split}
    \sum_{k=1}^{N} \frac{1}{k} e^{i \phi(k,x,t)} & = \sum_{k=1}^{N} \frac{1}{k} e^{i \rho(k,x,t)} e^{i \bar \phi(k,x,t)} = \sum_{k=1}^{N} \frac{1}{k} e^{i \rho(k,x,t)} (s(k,x,t) - s(k-1,x,t)) \\
    & = \sum_{k=1}^{N} \frac{1}{k} e^{i \rho(k,x,t)} s(k,x,t) - \sum_{k=0}^{N-1} \frac{1}{k+1} e^{i \rho(k+1,x,t)} s(k,x,t) \\
    & = \frac{1}{N} e^{i \rho(N,x,t)} s(N,x,t) - \sum_{k=1}^{N-1} \Big( \frac{1}{k+1} e^{i \rho(k+1,x,t)} - \frac{1}{k} e^{i \rho(k,x,t)} \Big) s(k,x,t) \, .
  \end{split}
  \]
  We estimate the terms in the above expression as follows.

  By Lemma~\ref{lem:series_exp}, the first term in~\eqref{eq:summation_by_parts_replacement} is bounded by
  \[ \label{eq:first_term_bound}
  \Big| \frac{1}{N} e^{i \rho(N,x,t)} s(N,x,t) \Big| \leq \frac{1}{N} \Big| \sum_{k=1}^N e^{i \bar \phi(k,x,t)} \Big| \leq C ( N^{-1} + t^{\frac{1}{1-\alpha}} N^{-1} + t^{-\frac{1}{2}} N^{- \frac{\alpha}{2}}  ) \, ,
  \]
  which, for a fixed $t > 0$, tends to zero as $N \to +\infty$ uniformly with respect to $x \in [-\frac{1}{2},\frac{1}{2}]$.

  To study the second term in~\eqref{eq:summation_by_parts_replacement}, we first observe that, by Lemma~\ref{lem:omega_prime},
  \begin{equation} \label{eq:bound_rho_derivative}
  |\de_\xi \rho(\xi,x,t)| = |\de_\xi \phi(\xi,x,t) - \de_\xi \bar \phi(\xi,x,t)| = |\omega'(\xi) - \alpha \gamma \xi^{\alpha-1} | t \leq C t \xi^{-\alpha - 1} \,,
  \end{equation}
  for $\xi > 0$, whence,
  \[
  \begin{split}
    \Big|  \frac{1}{k+1} e^{i \rho(k+1,x,t)} - \frac{1}{k} e^{i \rho(k,x,t)} \Big| & \leq \Big|\frac{e^{i \rho(k+1,x,t)} - e^{i \rho(k,x,t)}}{k+1} \Big| + \frac{1}{k(k+1)} \\
    & \leq \frac{1}{k+1} \sup_{\xi \in [k,k+1]} |\de_\xi \rho(\xi,x,t)| + \frac{1}{k^2} \\
    & \leq C\frac{t}{k^{\alpha+2}}+ \frac{1}{k^2} \leq C (1+t) k^{-2} \,.
  \end{split}
  \]
  This, together with Lemma~\ref{lem:series_exp}, yields
  \[ \label{eq:second_term_bound}
  \begin{split}
& \Big| \Big( \frac{1}{k+1} e^{i \rho(k+1,x,t)} - \frac{1}{k} e^{i \rho(k,x,t)} \Big) s(k,x,t) \Big|  \leq C (1+t)k^{-2} |s(k,x,t)| \\
    & \quad \leq C k^{-2} (1+t) (1+ t^{\frac{1}{1-\alpha}} + t^{-\frac{1}{2}} k^{1 - \frac{\alpha}{2}}  ) \leq C_t k^{-1 - \frac{\alpha}{2}} \, ,
\end{split}
  \]
 where $C_t > 0$ depends on $t$.
 In particular, given $t > 0$, the series
  \[
  \sum_{k=1}^{+\infty} \Big( \frac{1}{k+1} e^{i \rho(k+1,x,t)} - \frac{1}{k} e^{i \rho(k,x,t)} \Big) s(k,x,t)
  \]
  converges uniformly with respect to $x \in [-\frac{1}{2},\frac{1}{2}]$.

  By~\eqref{eq:summation_by_parts_replacement}, we conclude that the series
  \[
  \sum_{k=1}^{+\infty} \frac{1}{k} e^{i \phi(k,x,t)}
  \]
  converges uniformly with respect to $x \in [-\frac{1}{2},\frac{1}{2}]$.

  The bound~\eqref{eq:series_L_infty_bound} follows from~\eqref{eq:first_term_bound} and~\eqref{eq:second_term_bound}.
\end{proof}

\subsection{Proof of Theorem~\ref{thm:dissolution}} \label{subsec:proof-dissolution}

We are now in a position to prove the main result in this section.

\begin{proof}[Proof of Theorem~\ref{thm:dissolution}]
The unique solution to the Cauchy problem~\eqref{eq:peridynamics} with initial data $u_0$ and $v_0$ is given by the formula~\eqref{eq:solution_formula}.
By Remark~\ref{rem:u0}, in this specific case, we have that
\[
u(t,x) = \sum_{k \in \Z \setminus \{0\}} \frac{\cos(\omega(k) t)}{i 2 \pi k} e^{i 2 \pi k x} = \sum_{k \in \Z \setminus \{0\}} \frac{1}{i 4 \pi k} ( e^{i 2 \pi k x + i \omega(k) t} + e^{i 2 \pi k x - i \omega(k) t}) \, .
\]
Hence, the continuity of $u(t,\cdot)$ for $t > 0$ is implied by the uniform convergence of the oscillatory series
\[
\sum_{k \in \Z \setminus \{0\}} \frac{1}{k} e^{i 2\pi k x + i \omega(k) t} \, , \quad \text{for } t > 0 \, , \ x \in \T^1 \, ,
\]
the argument being analogous for the series with $-i \omega(k) t$.
This follows from Theorem~\ref{thm:series_convergence}.
The bound~\eqref{eq:L_infty_bound_solution} follows from~\eqref{eq:series_L_infty_bound}.
\end{proof}
\section{H\"older estimates} \label{sec:holder}

In this section, we quantify precisely the gain of H\"older regularity of the solution to the Cauchy problem~\eqref{eq:peridynamics} with initial data $u_0$ and $v_0$ as in Theorem~\ref{thm:dissolution}.

The main result of this section is the following.
It shows that the solution with the discontinuous initial data given by~\eqref{eq:initial_data} becomes instantaneously H\"older continuous with exponent $s$ for every $s \in (0,\frac{\alpha}{2}]$.
As $t \to 0^+$, the H\"older norm of the solution blows up at a precise rate, which is given by $t^{-\frac{s}{\alpha}}$, quantifying the loss of continuity as $t \to 0^+$.

\begin{theorem} \label{thm:hoelder_estimate}
  Let $\alpha \in (0,\frac{1}{2})$.
  Let $s \in (0, \frac{\alpha}{2}]$.
  Let $u_0 \in H^\alpha(\T^1)$ and $v_0 \in L^2(\T^1)$ be as in~\eqref{eq:initial_data}.
  Let $u \in C([0,+\infty);H^\alpha(\T^1)) \cap C^1([0,+\infty);L^2(\T^1))$ be the unique solution to the Cauchy problem~\eqref{eq:peridynamics} with initial data $u_0$ and $v_0$ provided by Theorem~\ref{thm:existence}.
  Then there exist $t_0 > 0$ and $C > 0$ such that
  \begin{equation} \label{eq:lower_bound_holder}
      \frac{1}{C} t^{-\frac{s}{\alpha}} \leq \| u(t, \cdot) \|_{C^{0,s}(\T^1)} \leq C t^{-\frac{s}{\alpha}} \, , \quad \text{ for every } t \in (0,t_0) \, .
  \end{equation}
\end{theorem}
\begin{proof}
    \step{1} (Upper bound)
    Let us prove that
\begin{equation*}
    \| u(t, \cdot) \|_{C^{0,s}(\T^1)} \leq C t^{-\frac{s}{\alpha}} \, .
\end{equation*}
We write the solution explicitly for $t > 0$ and $x \in \T^1$ as follows
\begin{equation*}
    u(t,x) = \sum_{k \in \Z \setminus \{0\}} \frac{\cos(\omega(k) t)}{i 2 \pi k} e^{i 2 \pi k x} = \sum_{k \in \Z \setminus \{0\}} \frac{1}{i 4 \pi k} \Big( e^{i 2 \pi k x + i \omega(k) t} + e^{i 2 \pi k x - i \omega(k) t} \Big) \, .
\end{equation*}
To measure the H\"older regularity of $u(t, \cdot)$, we use the Besov characterization of H\"older spaces given in Lemma~\ref{lem:holder_besov}.
Adopting the notation of the Littlewood-Paley decomposition introduced in Subsection~\ref{subsec:holder_besov}, for $j \geq 0$, $t > 0$, and $x \in \T^1$ we have that
\begin{equation*}
    P_j u(t,x) = \sum_{k \in \Z \sm \{0\}} \psi_j(|k|) \frac{1}{i 4 \pi k} \Big( e^{i 2 \pi k x + i \omega(k) t} + e^{i 2 \pi k x - i \omega(k) t} \Big) \, .
\end{equation*}
Our goal is to bound $\| P_j u(t, \cdot) \|_{L^\infty(\T^1)}$ from above for every $j \geq 0$.
We adopt the notation used in the proof of Theorem~\ref{thm:series_convergence} and we write, for $t > 0$ and $x \in \T^1$,
\begin{equation*}
  s(k,x,t) := \sum_{h=1}^k e^{i \bar \phi(h,x,t)} \, , \quad \rho(\xi,x,t) := \phi(\xi,x,t) - \bar \phi(\xi,x,t) \, .
\end{equation*}
Applying summation by parts, we get that, for $j \geq 2$, $t > 0$ and $x \in \T^1$,
\begin{equation} \label{eq:summation_by_parts_holder}
  \begin{split}
      \sum_{k = 1}^{+\infty} \psi_j(k) \frac{1}{k} e^{i \phi(k,x,t)} & = \sum_{k = 1}^{+\infty} \psi_j(k) \frac{1}{k} e^{i \rho(k,x,t)} e^{i \bar \phi(k,x,t)} \\
                                                                     & = \sum_{k = 1}^{+\infty} \psi_j(k) \frac{1}{k} e^{i \rho(k,x,t)} (s(k,x,t) - s(k-1,x,t)) \\
                                                                                           & = \sum_{k = 1}^{+\infty} \psi_j(k) \frac{1}{k} e^{i \rho(k,x,t)} s(k,x,t) -  \sum_{k = 0}^{+\infty} \psi_j(k+1) \frac{1}{k+1} e^{i \rho(k+1,x,t)} s(k,x,t) \\
                                                                                           & = - \sum_{k = 2^{j-1}-1}^{2^{j+1}} \Big( \psi_j(k+1) \frac{1}{k+1} e^{i \rho(k+1,x,t)} - \psi_j(k) \frac{1}{k} e^{i \rho(k,x,t)} \Big) s(k,x,t) \, .
  \end{split}
\end{equation}
Exploiting~\eqref{eq:bound_rho_derivative}, we estimate, for $2^{j-1} - 1 \leq k \leq 2^{j+1}$,
\begin{equation*}
  \begin{split}
    & \Big| \psi_j(k+1) \frac{1}{k+1} e^{i \rho(k+1,x,t)} - \psi_j(k) \frac{1}{k} e^{i \rho(k,x,t)} \Big| \\ 
    & \quad \leq \Big| \frac{\psi_j(k+1) - \psi_j(k)}{k+1} \Big| + \Big| \psi_j(k) \Big( \frac{1}{k+1} - \frac{1}{k} \Big) \Big| + | \psi_j(k) | \frac{1}{k} | e^{i \rho(k+1,x,t)} - e^{i \rho(k,x,t)} | \\
    & \quad \leq \sup_{h} |\psi'_j(h)| \frac{1}{k+1} + \sup_{h} |\psi_j(h)| \frac{1}{k(k+1)} +  \sup_{h} | \psi_j(h) | \frac{1}{k} \sup_{\xi \in [k,k+1]} | \de_\xi \rho(\xi,x,t) | \\
    & \quad \leq \sup_{h} |2^{-j} \psi'_0(2^{-j} h)| \frac{1}{k+1} + \sup_{h} |\psi_0(2^{-j} h)| \frac{1}{k(k+1)} + \sup_{h} | \psi_0(2^{-j} h) | C t k^{-\alpha-2} \\  
    & \quad \leq C 2^{-2j} + C t 2^{-(2+\alpha)j} \leq C 2^{-2j}  \, ,
  \end{split}
\end{equation*}
for $t \in (0,1)$.
By Lemma~\ref{lem:series_exp}, for $t \in (0,t_0)$ with $t_0 > 0$ chosen sufficiently small and $2^{j-1} - 1 \leq k \leq 2^{j+1}$, we have that
\begin{equation*}
    |s(k,x,t)| \leq C t^{-\frac{1}{2}} 2^{j(1 - \frac{\alpha}{2})} \, .
\end{equation*}
On the other hand, we have the trivial bound
\begin{equation*}
    |s(k,x,t)| \leq C 2^j \, .
\end{equation*}
Putting these estimates into~\eqref{eq:summation_by_parts_holder} and reasoning analogously for negative $k$'s, we get that, for $t \in (0,t_0)$ and $x \in \T^1$,
\begin{equation*}
|P_j u(t,x)| \leq C 2^j \cdot 2^{-2 j} \min\{t^{-\frac{1}{2}} 2^{j(1 - \frac{\alpha}{2})}, 2^j\} \leq C \min\{t^{-\frac{1}{2}} 2^{-j \frac{\alpha}{2}}, 1\} \, ,
\end{equation*}
for $j \geq 2$. 
Estimating $P_{-1} u$, $P_0 u$, $P_1$ separately with trivial bounds, we get
\begin{equation} \label{eq:P_ju_besov_bound}
    \| u(t, \cdot) \|_{B^s_{\infty,\infty}(\T^1)} = \sup_{j \geq -1} 2^{j s} \| P_j u(t, \cdot) \|_{L^\infty(\T^1)} \leq \sup_{j \geq -1} C 2^{j s} \min\{t^{-\frac{1}{2}} 2^{-j \frac{\alpha}{2}}, 1\} \, .
\end{equation}
Let us compute the supremum in the right-hand side of~\eqref{eq:P_ju_besov_bound} by discussing the two cases $t^{-\frac{1}{2}} 2^{-j \frac{\alpha}{2}} \leq 1$ and $t^{-\frac{1}{2}} 2^{-j \frac{\alpha}{2}} \geq 1$.
If $j$ is such that $t^{-\frac{1}{2}} 2^{-j \frac{\alpha}{2}} \leq 1$, \ie,
\begin{equation} \label{eq:j_lower_bound_split_holder}
    \Big\lceil -\frac{1}{\alpha} \log_2 ( t ) \Big\rceil \leq j \, ,
\end{equation}
then $2^{j s} \min\{t^{-\frac{1}{2}} 2^{-j \frac{\alpha}{2}}, 1\} = t^{-\frac{1}{2}} 2^{j(s - \frac{\alpha}{2})}$, which is nonincreasing in $j$, since $s \in (0, \frac{\alpha}{2}]$.
Hence, for such $j$'s, the supremum is attained with equality in~\eqref{eq:j_lower_bound_split_holder}, and gives
\begin{equation*}
    C 2^{j s} \min\{t^{-\frac{1}{2}} 2^{-j \frac{\alpha}{2}}, 1\} \leq C t^{-\frac{1}{2}} 2^{- \frac{1}{\alpha} \log_2 ( t ) (s - \frac{\alpha}{2})} = C t^{-\frac{s}{\alpha}} \, .
\end{equation*}
If $j$ is such that $t^{-\frac{1}{2}} 2^{-j \frac{\alpha}{2}} \geq 1$, \ie,
\begin{equation} \label{eq:j_upper_bound_split_holder}
    \Big\lfloor -\frac{1}{\alpha} \log_2 ( t ) \Big\rfloor \geq j \, ,
\end{equation}
then $2^{j s} \min\{t^{-\frac{1}{2}} 2^{-j \frac{\alpha}{2}}, 1\} = 2^{j s}$, which is nondecreasing in $j$, since $s \in (0,\frac{\alpha}{2}]$.
The supremum is attained with equality in~\eqref{eq:j_upper_bound_split_holder}, and gives
\begin{equation*}
    C 2^{j s} \min\{t^{-\frac{1}{2}} 2^{-j \frac{\alpha}{2}}, 1\} \leq C 2^{j s} \leq C 2^{ - \frac{1}{\alpha} \log_2 ( t )s} = C t^{-\frac{s}{\alpha}} \, .
\end{equation*}
In either case, from~\eqref{eq:P_ju_besov_bound} and Lemma~\ref{lem:holder_besov}, we deduce that
\begin{equation*}
    \| u(t, \cdot) \|_{C^{0,s}(\T^1)} \leq C \| u(t, \cdot) \|_{B^s_{\infty,\infty}(\T^1)} \leq C t^{-\frac{s}{\alpha}} \, ,
\end{equation*}
concluding the proof of the upper bound.

\step{2} (Lower bound)
    Let us prove that
    \begin{equation*}
        \| u(t, \cdot) \|_{C^{0,s}(\T^1)} \geq C t^{-\frac{s}{\alpha}} \, .
    \end{equation*}
    The proof will actually hold true for every $s \in (0,1)$.
We start by writing the solution explicitly as follows
\begin{equation*}
    u(t,x) = \sum_{k \in \Z \setminus \{0\}} \frac{\cos(\omega(k) t)}{i 2 \pi k} e^{i 2 \pi k x} = \frac{1}{\pi} \sum_{k = 1}^{+\infty}  \frac{1}{k} \cos(\omega(k) t) \sin(2 \pi k x) \, .
\end{equation*}
We will measure the H\"older regularity of $u(t, \cdot)$ using Lemma~\ref{lem:holder_besov}.
For every $j \geq 1$, we have that,
\begin{equation} \label{eq:P_ju_explicit}
    P_j u(t,x) = \sum_{k \in \Z} \psi_j(|k|) \hat u_k e^{i 2 \pi k x} = \frac{1}{\pi} \sum_{k = 2^{j-1}}^{2^{j+1}} \psi_j(k) \frac{1}{k} \cos(\omega(k) t) \sin(2 \pi k x) \, .
\end{equation}

By Remark~\ref{rem:omega_bounds}, we have that $\omega(k) \leq C_2 |k|^\alpha$ for some $C_2 > 0$.
Given $t > 0$, we choose
\begin{equation} \label{eq:j_choice_t}
 j := \Big\lfloor \frac{1}{\alpha} \log_2 \Big( \frac{\pi}{3 C_2 t} \Big) - 1 \Big\rfloor \, ,
\end{equation}
so that
\begin{equation*}
    j \leq \frac{1}{\alpha} \log_2 \Big( \frac{\pi}{3 C_2 t} \Big) - 1 \implies C_2 2^{\alpha (j+1)} t \leq  \frac{\pi}{3}\, ,
\end{equation*}
and
\begin{equation} \label{eq:j_choice_t_2}
    \frac{1}{\alpha} \log_2 \Big( \frac{\pi}{3 C_2 t} \Big) - 1 < j + 1 \implies  2^{ j } > \frac{1}{4}\Big(\frac{\pi}{3 C_2}\Big)^\frac{1}{\alpha} t^{-\frac{1}{\alpha}} \geq C t^{-\frac{1}{\alpha}} \, .
\end{equation}
Note that $j \geq 1$ for $t \in (0, t_0)$ with $t_0 > 0$ chosen sufficiently small.
This choice implies that, for $2^{j-1} \leq k \leq 2^{j+1}$,
\begin{equation*}
    0 < \omega(k) t \leq C_2 k^\alpha t \leq C_2 2^{\alpha(j+1)} t \leq \frac{\pi}{3} \, ,
\end{equation*}
and thus
\begin{equation*}
    \cos(\omega(k) t) \geq \frac{1}{2} \, , \quad \text{ for }  2^{j-1} \leq k \leq 2^{j+1} \, .
\end{equation*}
Given $j$ as in~\eqref{eq:j_choice_t}, we choose
\begin{equation*}
    x_j := \frac{9}{48} 2^{-j} \, ,
\end{equation*}
so that
\begin{equation*}
    \frac{\pi}{6} < \frac{9}{48} \pi = 2 \pi 2^{j-1} x_j \leq 2 \pi k x_j \leq 2 \pi 2^{j+1} x_j = \frac{36}{48} \pi < \frac{5 \pi}{6} \, ,
\end{equation*}
and thus
\begin{equation*}
    \sin(2 \pi k x_j) \geq \frac{1}{2} \, , \quad \text{ for } 2^{j-1} \leq k \leq 2^{j+1} \, .
\end{equation*}
Finally, by~\eqref{eq:P_ju_explicit}, applying Lemma~\ref{lem:psi_lower_bound}, we get that,
\begin{equation*}
  \begin{split}
      \| P_j u(t, \cdot) \|_{L^\infty(\T^1)} & \geq P_j u(t, x_j) \geq \frac{1}{\pi} \sum_{k=2^{j-1}}^{2^{j+1}} \psi_j(k) \frac{1}{k} \cdot \frac{1}{2} \cdot \frac{1}{2} \\
                                             & \geq \frac{1}{\pi} \sum_{k=k'_j}^{k''_j} \frac{1}{k} \cdot \frac{1}{2} \cdot \frac{1}{2} \cdot \frac{1}{2} \geq \frac{1}{8 \pi} \sum_{k=k'_j}^{k''_j} \frac{1}{k} \geq C \log\Big( \frac{k''_j}{k'_j} \Big)\, ,
  \end{split}
\end{equation*}
where
\begin{equation*}
    k'_j = \Big\lceil 2^{j-1} \Big( 1 + \Big( \frac{\log(2)}{1 + \log(2)} \Big)^\frac{1}{2} \Big) \Big\rceil \, , \quad k''_j  = \Big\lfloor 2^{j} \Big( 1 + \Big( \frac{\log(2)}{1 + \log(2)} \Big)^\frac{1}{2} \Big) \Big\rfloor \, ,
\end{equation*}
are the indices provided by Lemma~\ref{lem:psi_lower_bound}.
Note that, for $j \geq 1$,
\begin{equation*}
    \frac{k''_j}{k'_j} = \frac{ \Big\lfloor 2^{j} \Big( 1 + \Big( \frac{\log(2)}{1 + \log(2)} \Big)^\frac{1}{2} \Big) \Big\rfloor }{\Big\lceil 2^{j-1} \Big( 1 + \Big( \frac{\log(2)}{1 + \log(2)} \Big)^\frac{1}{2} \Big) \Big\rceil} \geq C > 1 \, ,
\end{equation*}
for some constant $C$ independent of $j$, and thus $\| P_j u(t, \cdot) \|_{L^\infty(\T^1)} \geq C$.

By Lemma~\ref{lem:holder_besov} and by~\eqref{eq:j_choice_t_2}, it follows that
\begin{equation*}
    \| u(t, \cdot) \|_{C^{0,s}(\T^1)} \geq C \| u(t, \cdot) \|_{B^s_{\infty,\infty}(\T^1)} \geq C 2^{j s} \| P_j u(t, \cdot) \|_{L^\infty(\T^1)} \geq C 2^{j s} \geq C t^{-\frac{s}{\alpha}}\, .
\end{equation*}
This concludes the proof of~\eqref{eq:lower_bound_holder}.
\end{proof}

\section{An example of nucleation of a jump discontinuity} \label{sec:nucleation}

In this section, we provide an example of a continuous initial datum $u_0$ and an initial velocity $v_0$ such that the corresponding peridynamic evolution develops a jump discontinuity at a given later time $t_1 > 0$.
The construction simply relies on the reversibility of the linear peridynamic equation and on the dissolution of jump discontinuities proved in Theorem~\ref{thm:dissolution}.

\subsection{Main result about the nucleation of a jump discontinuity}
We let $t_1 > 0$ be fixed.
We let
\[
u_1(x) := \frac{1}{2} - x \, , \quad v_1(x) := 0 \, , \quad x \in [0,1) \, ,
\]
extended by 1-periodicity to the whole real line.
Notice that these are precisely the initial data considered in Section~\ref{sec:dissolution}.
We let $w \in C([0,+\infty);H^\alpha(\T^1)) \cap C^1([0,+\infty);L^2(\T^1))$ be the solution to the Cauchy problem~\eqref{eq:peridynamics} with initial data $u_1$ and $v_1$ provided by Theorem~\ref{thm:existence}.
By Theorem~\ref{thm:dissolution}, we know that $w(t,\cdot)$ is a continuous function for every $t > 0$.
We extend $w$ to negative times by setting
\[\label{eq:extension_negative_times}
w(t,x) := w(-t,x) \, \quad \text{for } t < 0 \, , \ x \in \R \, .
\]

Let us set
\[
u(t,x) := w(t_1 - t,x) \, , \quad \text{for } t > 0 \, , \ x \in \R \, .
\]
We observe that $u$ is a solution (thus the unique one, by Theorem~\ref{thm:uniqueness}) to the Cauchy problem~\eqref{eq:peridynamics} with initial data
\[
u_0(x) := w(t_1,x) \, , \quad v_0(x) := -\de_t w(t_1,x) \, , \quad x \in [0,1] \, ,
\]
extended by 1-periodicity to the whole real line.
Note that $u_0 \in H^\alpha(\T^1)$, $v_0 \in L^2(\T^1)$, and $u_0$ is a continuous function.
Finally, $u(t_1,x) = w(0,x) = u_1(x)$, which has a jump discontinuity.

Together with the bounds provided in Theorem~\ref{thm:hoelder_estimate}, this proves Theorem~\ref{thm:nucleation}.

\section{An example of a solution with a jump discontinuity on a dense set of times} \label{sec:dense_set}

Building on the example in Section~\ref{sec:nucleation}, we can construct a solution to the Cauchy problem~\eqref{eq:peridynamics} that nucleates spatial discontinuities at a sequence of times that is dense in the time interval $[0,1]$.

The result is the following.

\begin{theorem} \label{thm:dense_set}
 Let $\alpha \in (0,\frac{1}{2})$.
  There exist $u_0 \in H^\alpha(\T^1)$ and $v_0 \in L^2(\T^1)$ such that $u_0$ is a continuous function and the unique solution $u \in C([0,+\infty);H^\alpha(\T^1)) \cap C^1([0,+\infty);L^2(\T^1))$ to the Cauchy problem~\eqref{eq:peridynamics} with initial data $u_0$ and $v_0$ provided by Theorem~\ref{thm:existence} satisfies that $u(t,\cdot)$ is a discontinuous function for all $t \in D$ (more precisely, it is continuous for $x \neq 0$ and has a jump discontinuity at $x = 0$), where $D$ is a dense subset of $(0,1)$.
\end{theorem}
\begin{proof}
{\itshape Definition of the dense set.} We define $D$ as the sequence of dyadic rational times in $(0,1)$.
To order them, we consider the $n$-th generation of dyadic numbers in $(0,1)$
\[
D_n := \Big\{ \frac{j}{2^n} : j = 1, \ldots, 2^n - 1 \Big\} \, , \quad n \in \N \, , \ n \geq 1 \, , \quad D_0 := \emptyset \, .
\]
Then we set
\[
D := \bigcup_{n=1}^{+\infty} D_n \, ,
\]
which is a dense subset of $(0,1)$.

{\itshape Definition of the function $u$.} We define the function $u$ as follows.
As done in Theorem~\ref{thm:nucleation}, for all $s \in [0,1]$, we consider the solution that nucleates a jump discontinuity at time $t_1 = s$ given by $w(s - t,x)$, where $w$ is the solution provided by Theorem~\ref{thm:dissolution} extended to negative times as in~\eqref{eq:extension_negative_times}.
The function $w(s-t,\cdot)$ is discontinuous in $x$ only at time $s$.
We set
\[
u(t,x) := \sum_{n=1}^{+\infty} a_n  \sum_{s \in D_n \sm D_{n-1}} w(s-t,x) \, ,
\]
where
\[ \label{eq:cn}
a_n := 2^{-3n} \, , \quad n \in \N \, , \ n \geq 1 \, .
\]
This series is intended as the limit of the sequence of partial sums
\[ \label{eq:partial_sums}
u^N(t,x) := \sum_{n=1}^{N} a_n  \sum_{s \in D_n \sm D_{n-1}} w(s-t,x) \, ,
\]
as $N \to +\infty$.
We observe that $u$ is well-defined. Indeed, for $N \geq M \geq 1$, by Theorem~\ref{thm:energy_conservation}, and by~\eqref{eq:cn}, we have that
\[
\begin{split}
    [ u^N(t,\cdot) - u^M(t,\cdot)]_{\W^\alpha_\delta(\T^1)} & = \Big[ \sum_{n=M+1}^{N} a_n  \sum_{s \in D_n \sm D_{n-1}} w(s-t,\cdot) \Big]_{\W^\alpha_\delta(\T^1)} \\
                                                    & \leq \sum_{n=M+1}^{N} a_n  \sum_{s \in D_n \sm D_{n-1}} [ w(s-t,\cdot) ]_{\W^\alpha_\delta(\T^1)} \\
                                                    & \leq C \sum_{n=M+1}^{N} a_n 2^{n-1} [ u_1 ]_{\W^\alpha_\delta(\T^1)} \leq C [ u_1 ]_{\W^\alpha_\delta(\T^1)} \sum_{n=M+1}^{N} 2^{-2n} \\
    & \leq C 2^{-2M} \, .
\end{split}
\]
Since $\int_{\T^1}u^N(t,x) \d x = 0$, by the Poincaré inequality and by Proposition~\ref{prop:W_equiv_H}, we deduce that $u^N(t,\cdot)$ is a Cauchy sequence in $H^\alpha(\T^1)$, uniformly with respect to $t \in [0,1]$, thus defining $u \in C([0,1];H^\alpha(\T^1))$ as the limit of $u^N$ as $N \to +\infty$.
We observe that $u \in C^1([0,1];L^2(\T^1))$ as well.
Indeed, we have that
\[
\begin{split}
    \| \de_t u^N(t,\cdot) - \de_t u^M(t,\cdot)\|_{L^2(\T^1)} & = \Big\| \sum_{n=M+1}^{N} a_n  \sum_{s \in D_n \sm D_{n-1}} \de_t w(s-t,\cdot) \Big\|_{L^2(\T^1)} \\
    & \leq \sum_{n=M+1}^{N} a_n  \sum_{s \in D_n \sm D_{n-1}} \| \de_t w(s-t,\cdot) \|_{L^2(\T^1)} \\
    & \leq C \sum_{n=M+1}^{N} a_n 2^{n-1} \| u_1 \|_{H^\alpha(\T^1)} \leq C \| u_1 \|_{H^\alpha(\T^1)} \sum_{n=M+1}^{N} 2^{-2n} \\
    & \leq C 2^{-2M} \, ,
\end{split}
\]
showing that $\de_t u^N(t,\cdot)$ is a Cauchy sequence in $L^2(\T^1)$, uniformly with respect to $t \in [0,1]$, thus defining a limit $v \in C([0,1];L^2(\T^1))$ as $N \to +\infty$.
By writing $u^N$ as the Bochner integral in $L^2(\T^1)$,
\[
u^N(t,\cdot) = u^N(0,\cdot) + \int_0^t \de_t u^N(\tau,\cdot) \, \d \tau \, ,
\]
we can pass to the limit as $N \to +\infty$ and we obtain that $u$ is differentiable in $L^2(\T^1)$ with $\de_t u(t,\cdot) = v(t,\cdot)$.

{\itshape Proof that $u$ is a solution.} First of all, we observe that $u^N$ defined in~\eqref{eq:partial_sums} are solutions to the peridynamics equation  by linearity.
Their convergence to $u$ as $N \to +\infty$ is enough to pass to the limit in the weak formulation of the peridynamics equation
\[
\begin{split}
& \int_0^{1} \int_{\T^1} u^N(t,x) \de_{tt} \varphi(t,x) \, \d x \, \d t - \int_0^{1} \int_{\T^1} u^N(t,x) K[\varphi(t,\cdot)](x) \, \d x \, \d t\\
& \quad  = - \int_{\T^1} u^N(0,x) \de_t \varphi(0,x) \, \d x + \int_{\T^1} \partial_t u^N(0,x) \varphi(0,x) \, \d x \, ,
\end{split}
\]
for every $\varphi \in C_c^\infty([0,1) \times \T^1)$.
By Theorem~\ref{thm:existence} there exists a global solution defined for all times $t \in [0,+\infty)$ with initial data $u(0,\cdot)$ and $\partial_t u(0,\cdot)$. 
By Theorem~\ref{thm:uniqueness}, we conclude that it must coincide with $u(t,\cdot)$ for $t \in [0,1]$.

{\itshape Proof of discontinuity.} Let us fix $\bar s \in D$ and let us show that $u(\bar s,\cdot)$ is a discontinuous function.
We let $\bar n \in \N$, $\bar n \geq 1$, be such that $\bar s \in D_{\bar n} \sm D_{\bar n-1}$.
We write
\[
\begin{split}
    u(\bar s,x) & = \sum_{n=1}^{\bar n-1} a_n  \sum_{s \in D_n \sm D_{n-1}} w(s-\bar s,x) \\
    & \quad + a_{\bar n} w(0,x) + a_{\bar n} \sum_{\substack{s \in D_{\bar n} \sm D_{\bar n-1} \\ s \neq \bar s}}  w(s-\bar s,x) \\
    & \quad + \sum_{n=\bar n+1}^{+\infty} a_n  \sum_{s \in D_n \sm D_{n-1}} w(s-\bar s,x) \, .
\end{split}
\]
The function $w(0, \cdot)$ is discontinuous by construction.
The functions
\[
\sum_{n=1}^{\bar n-1} a_n  \sum_{s \in D_n \sm D_{n-1}} w(s-\bar s,x) \quad \text{and} \quad a_{\bar n} \sum_{\substack{s \in D_{\bar n} \sm D_{\bar n-1} \\ s \neq \bar s}}  w(s-\bar s,x) \, ,
\]
are finite sums of continuous functions, thus they are continuous.
The only nontrivial part is to show that the series
\[ \label{eq:series_from_n_bar}
\sum_{n=\bar n+1}^{+\infty} a_n  \sum_{s \in D_n \sm D_{n-1}} w(s-\bar s,x)
\]
converges uniformly in $x$, thus defining a continuous function.
To this end, we observe that, by construction and by Theorem~\ref{thm:dissolution}, we have that
\[
\|w(s-\bar s,\cdot)\|_{L^\infty(\T^1)} \leq C |s - \bar s|^{-\frac{1}{2}} \, .
\]
However, for $n > \bar n$, we have that $s \neq \bar s$ and thus $|s - \bar s| \geq 2^{-n}$.
This yields
\[
\|w(s-\bar s,\cdot)\|_{L^\infty(\T^1)} \leq C 2^{\frac{n}{2}} \, , \quad \text{for every } s \in D_n \sm D_{n-1} \, , \ n > \bar n \, .
\]
There are $2^{n-1}$ elements in $D_n \sm D_{n-1}$, hence by the choice of $a_n$ in~\eqref{eq:cn}, we estimate
\[
    \Big\| a_n \sum_{s \in D_n \sm D_{n-1}} w(s-\bar s,\cdot) \Big\|_{L^\infty(\T^1)} \leq C a_n 2^{n-1} 2^{\frac{n}{2}} = C a_n 2^{\frac{3n}{2} - 1} \leq C 2^{-\frac{3n}{2}} \, ,
\]
for $n > \bar n$.
This shows that the series in~\eqref{eq:series_from_n_bar} converges uniformly in $x$, concluding the proof that $u(\bar s,\cdot)$ is discontinuous.

{\itshape Proof of continuity at $t=0$}. Finally, let us show that $u(0,\cdot)$ is a continuous function.
By Theorem~\ref{thm:dissolution}, we have that 
\begin{equation*}
    \|w(s,\cdot)\|_{L^\infty(\T^1)} \leq C |s|^{-\frac{1}{2}} \, .
\end{equation*}
Given $n \geq 1$ and $s \in D_n \setminus D_{n-1}$, we have that $s \geq 2^{-n}$, yielding 
\begin{equation*}
    \| w(s,\cdot) \|_{L^\infty(\T^1)} \leq C 2^{\frac{n}{2}} \, .
\end{equation*}
Consequently,
\begin{equation*}
    \Big\| a_n \sum_{s \in D_n \setminus D_{n-1}} w(s,\cdot) \Big\|_{L^\infty(\T^1)} \leq C 2^{-\frac{3n}{2}} \, , 
\end{equation*}
implying uniform convergence of the series.
\end{proof}

\appendix

\section{Proof of Lemma~\ref{lem:holder_besov}} \label{app:holder_besov}

We provide here the proof of Lemma~\ref{lem:holder_besov}.

\begin{proof}
The arguments in the proof are standard, see, \eg, \cite[Theorem~2.36]{BahCheDan13}.
We provide here the details for the reader's convenience.

\emph{Step 1}: Let us fix $u \in C^{0,s}(\T^1)$ and let us prove that $u \in B^s_{\infty,\infty}(\T^1)$ and
\begin{equation} \label{eq:besov_less_holder}
\| u \|_{B^s_{\infty,\infty}(\T^1)} \leq C \| u \|_{C^{0,s}(\T^1)} \, .
\end{equation}

We start by observing that
\begin{equation} \label{eq:kernel_representation}
  \begin{split}
      P_j u(x) & = \sum_{k \in \Z} \psi_j(|k|) \hat{u}_k e^{i 2 \pi k x} \\
               & = \int_{\T^1} u(y) \Big( \sum_{k \in \Z} \psi_j(|k|) e^{i 2 \pi k (x-y)} \Big) \d y = \int_{\T^1} u(y) \Psi_j(x-y) \d y \, ,
  \end{split}
\end{equation}
where
\begin{equation*}
    \Psi_j(z) := \sum_{k \in \Z} \psi_j(|k|) e^{i 2 \pi k z} \, .
\end{equation*}
Note that $\psi_j(0) = 0$ for $j \geq 0$, hence
\begin{equation*}
    \int_{\T^1} \Psi_j(z) \d z = 0  \quad \text{for } j \geq 0 \, .
\end{equation*}
This implies that
\begin{equation} \label{eq:bound_on_dyadic_block}
  \begin{split}
      | P_j u(x) | & = \Big| \int_{\T^1} u(y) \Psi_j( x - y ) \, \d y \Big| = \Big| \int_{\T^1} (u(y) - u(x)) \Psi_j( x - y ) \, \d y \Big| \\
                   & \leq \int_{\T^1} | u(y) - u(x) | | \Psi_j( x - y ) | \, \d y =  \int_{-\frac{1}{2}}^{\frac{1}{2}} | u(x-z) - u(x) | | \Psi_j( z ) | \, \d y \\
                   & \leq [u]_{C^{0,s}(\T^1)} \int_{-\frac{1}{2}}^{\frac{1}{2}} | z |^s | \Psi_j( z ) | \, \d z  \, .
  \end{split}
\end{equation}
To estimate the last term, we start by deducing a pointwise bound for $\Psi_j$.
On the one hand, we use the fact that $\psi_j$ is supported in $(2^{j-1}, 2^{j+1})$ to get that
\begin{equation*}
  \begin{split}
      (1-e^{-i 2 \pi z})^2 \sum_{k=0}^{+\infty} \psi_j(k) e^{i 2 \pi k z} & = \sum_{k=0}^{+\infty} \psi_j(k) e^{i 2 \pi k z} - 2 \sum_{k=0}^{+\infty} \psi_j(k) e^{i 2 \pi (k-1) z} + \sum_{k=0}^{+\infty} \psi_j(k) e^{i 2 \pi (k-2)z} \\
                                                                        & = \sum_{k=0}^{+\infty} \psi_j(k) e^{i 2 \pi k z} - 2 \sum_{k=-1}^{+\infty} \psi_j(k+1) e^{i 2 \pi k z} + \sum_{k=-2}^{+\infty} \psi_j(k+2) e^{i 2 \pi k z} \\
                                                                        & = \sum_{k=2^{j-1}-2}^{2^{j+1}} (\psi_j(k+2) - 2 \psi_j(k+1) + \psi_j(k)) e^{i 2 \pi k z} \, ,
\end{split}
\end{equation*}
for $j \geq 2$.
Reasoning analogously for negative $k$'s, we get that, for $j \geq 2$,
\begin{equation*}
    \Psi_j(z) = \frac{1}{(1-e^{-i 2 \pi z})^2} \sum_{2^{j-1}-2 \leq |k| \leq 2^{j+1}+2} (\psi_j(|k+2|) - 2 \psi_j(|k+1|) + \psi_j(|k|)) e^{i 2 \pi k z} \, .
\end{equation*}
Exploiting that
\begin{equation*}
    |\psi_j(|k+2|) - 2 \psi_j(|k+1|) + \psi_j(|k|)| \leq 2 \mathrm{Lip}(\psi_j') = 2 \cdot 2^{-2j} \mathrm{Lip}(\psi_0') \leq C 2^{-2j} \, ,
\end{equation*}
and that
\begin{equation} \label{eq:bound_on_denominator}
    |1 - e^{-i 2 \pi z}| = 2 |\sin(\pi z)| \geq 4 |z| \quad \text{for } |z| \leq \frac{1}{2} \, ,
\end{equation}
we deduce that
\begin{equation*}
    |\Psi_j(z)| \leq C 2^{-j} |z|^{-2} \, , \quad \text{for } j \geq 2 \text{ and } |z| \leq \frac{1}{2} \, .
\end{equation*}
On the other hand, we have the trivial bound
\begin{equation*}
    |\Psi_j(z)| \leq C 2^{j} \, , \quad \text{for } j \geq 0 \text{ and } |z| \leq \frac{1}{2} \, .
\end{equation*}
The previous bounds allow us to estimate:
\begin{equation*}
  \begin{split}
      \int_{\T^1} | z |^s | \Psi_j( z ) | \, \d z & = \int_{ \{ |z| \leq 2^{-j} \} } | z |^s | \Psi_j( z ) | \, \d z + \int_{ \{ 2^{-j} < |z| \leq \frac{1}{2} \} } | z |^s | \Psi_j( z ) | \, \d z \\
                                                  & \leq \int_{ \{ |z| \leq 2^{-j} \} } 2^{-js} 2^j \, \d z + C \int_{\{ 2^{-j} < |z| \leq \frac{1}{2} \}} |z|^{s} 2^{-j} |z|^{-2} \, \d z \\
                                                  & \leq 2^{-js} + C 2^{-j} ( 2^{j-js} - 2^{1-s} ) \leq C 2^{-js} \, ,
  \end{split}
\end{equation*}
for $j \geq 2$.
Inserting this estimate into~\eqref{eq:bound_on_dyadic_block}, we get that
\begin{equation*}
    \sup_{j \geq 2} 2^{s j} \| P_j u \|_{L^\infty(\T^1)} \leq C [u]_{C^{0,s}(\T^1)} \, .
\end{equation*}
We estimate the dyadic blocks $P_{-1} u$, $P_0 u$, and $P_1 u$ with trivial bounds.
Putting all the previous estimates together, we have shown~\eqref{eq:besov_less_holder}.

\emph{Step 2}: Let us fix $u \in B^s_{\infty,\infty}(\T^1)$. 
First of all, we observe that, by definition of the Besov norm, 
\begin{equation*}
    \sum_{j \geq -1} \| P_j u \|_{L^\infty(\T^1)} \leq \sum_{j \geq -1} 2^{-sj} \| u \|_{B^s_{\infty,\infty}(\T^1)} \leq C  \| u \|_{B^s_{\infty,\infty}(\T^1)} \, , 
\end{equation*}
hence the series in~\eqref{eq:LP_decomposition} converges uniformly to the continuous representative of $u$ (that we do not relabel). 

Let us now prove that
\begin{equation*}
    \| u \|_{C^{0,s}(\T^1)} \leq C \| u \|_{B^s_{\infty,\infty}(\T^1)} \, .
\end{equation*}

We start by estimating, by~\eqref{eq:LP_decomposition} (noticing that the averages cancel out),
\begin{equation} \label{eq:splitting_holder_estimate}
    | u(x) - u(y) | \leq \sum_{ j \geq 0 } | P_j u(x) - P_j u(y) | = \sum_{ 0 \leq j \leq \hat j} | P_j u(x) - P_j u(y) | + \sum_{ j > \hat j } | P_j u(x) - P_j u(y) | \, ,
\end{equation}
where $\hat j$ will be chosen later.
We have
\begin{equation} \label{eq:holder_estimate_on_dyadic_block}
    | P_j u(x) - P_j u(y) | \leq \| (P_j u)' \|_{L^\infty(\T^1)} |x-y| \, ,
\end{equation}
so the goal is to deduce a bound on $\| (P_j u)' \|_{L^\infty(\T^1)}$.
We claim that
\begin{equation} \label{eq:bernstein_inequality}
    \| (P_j u)' \|_{L^\infty(\T^1)} \leq C 2^{j} \| P_j u \|_{L^\infty(\T^1)} \, , \quad \text{for every } j \geq 0 \, .
\end{equation}
Let us prove the previous inequality.
By~\eqref{eq:kernel_representation}, for every $j \geq 0$,
\begin{equation*}
    (P_j u)'(x) =  \sum_{k \in \Z} i 2 \pi k \psi_j(|k|) \hat u_k e^{i 2 \pi k x} \, .
\end{equation*}
We introduce an auxiliary bump function $\tilde \psi_j$ which equals 1 on the support of $\psi_j$, see Figure~\ref{fig:auxiliary_bump}.
We provide a more precise definition of $\tilde \psi_j$ since we need a control on its derivative later in the proof.
We let $\tilde \psi_0 \in C^\infty_c(\R)$ be such that $\tilde \psi_0(\xi) = 1$ for $\frac{1}{2} \leq |\xi| \leq 2$ and $\tilde \psi_0(\xi) = 0$ for $|\xi| \leq \frac{1}{4}$ or $|\xi| \geq 4$.
We set $\tilde \psi_j(\xi) := \tilde \psi_0(2^{-j} \xi)$ and we observe that $\tilde \psi_j(|\xi|) = 1$ for $2^{j-1} \leq |\xi| \leq 2^{j+1}$, \ie, on the support of~$\psi_j$, and $\tilde \psi_j(|\xi|) = 0$ for $|\xi| \leq 2^{j-2}$ or $|\xi| \geq 2^{j+2}$.

\begin{figure}[ht]
    \begin{tikzpicture}[xscale=1.5]
        % picture for j = 1
        \pgfmathsetmacro{\j}{1}
        \draw ({2^(\j-2) - 0.1},0) -- ({2^(\j+2) + 0.1},0);

        \draw ({2^(\j-2)}, -0.1) -- ({2^(\j-2)}, 0.1);
        \draw ({2^(\j-2)}, -0.1) node[anchor=north] {$2^{j-2}$};

        \draw ({2^(\j-1)}, -0.1) -- ({2^(\j-1)}, 0.1);
        \draw ({2^(\j-1)}, -0.1) node[anchor=north] {$2^{j-1}$};

        \draw ({2^(\j)}, -0.1) -- ({2^(\j)}, 0.1);
        \draw ({2^(\j)}, -0.1) node[anchor=north] {$2^{j}$};

        \draw ({2^(\j+1)}, -0.1) -- ({2^(\j+1)}, 0.1);
        \draw ({2^(\j+1)}, -0.1) node[anchor=north] {$2^{j+1}$};

        \draw ({2^(\j+2)}, -0.1) -- ({2^(\j+2)}, 0.1);
        \draw ({2^(\j+2)}, -0.1) node[anchor=north] {$2^{j+2}$};

        \draw[thick, domain={2^(\j-1)}:{2^(\j)}, smooth, variable=\x] plot (\x, {1 - bump(2^(-\j) * \x)});
        \draw[thick, domain={2^(\j)}:{2^(\j+1)}, smooth, variable=\x] plot (\x, {bump(2^(-\j-1) * \x)});
        \draw ({2^(\j)}, {1-0.1}) node[anchor=north] {$\psi_j$};

        \draw[thick, dashed, domain={2^(\j-2)}:{2^(\j-1)}, smooth, variable=\x] plot (\x, {1 - bump(2^(-\j+1) * \x)});
        \draw[thick, dashed, domain={2^(\j-1)}:{2^(\j+1)}, smooth, variable=\x] plot (\x, 1);
        \draw[thick, dashed, domain={2^(\j+1)}:{2^(\j+2)}, smooth, variable=\x] plot (\x, {bump(2^(-\j-2) * \x)});
        \draw ({2^(\j+1)}, {1+0.1}) node[anchor=south] {$\tilde \psi_j$};
    \end{tikzpicture}

    \caption{The dashed line represents the graph of the auxiliary bump function $\tilde \psi_j$, while the solid line represents the graph of $\psi_j$.}
    \label{fig:auxiliary_bump}
\end{figure}
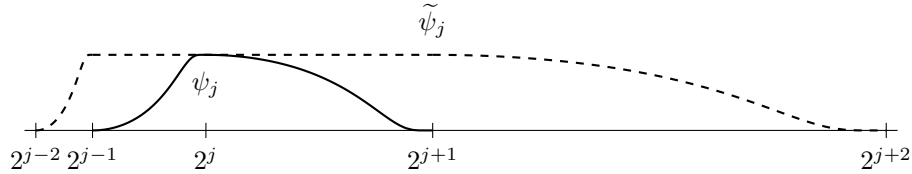

Then, we can write
\begin{equation} \label{eq:derivative_as_convolution}
    (P_j u)'(x) =  \sum_{k \in \Z} i 2 \pi k \tilde \psi_j(|k|) \psi_j(|k|) \hat u_k e^{i 2 \pi k x} = M_j * (P_j u)(x) \, ,
\end{equation}
where $M_j$ is the kernel defined by
\begin{equation*}
    M_j(z) := \sum_{k \in \Z} i 2 \pi k \tilde \psi_j(|k|) e^{i 2 \pi k z} \, .
\end{equation*}
We estimate $M_j$ as follows with arguments similar to the ones used to estimate $\Psi_j$ in Step~1. Specifically, for $j \geq 3$, we have that
\begin{equation*}
  \begin{split}
& ( 1 - e^{- i 2 \pi z})^2 \sum_{k = 0}^{+\infty} i 2 \pi k \tilde \psi_j(k) e^{i 2 \pi k z} \\
& \quad = \sum_{k=0}^{+\infty} i 2 \pi k \tilde \psi_j(k) e^{i 2 \pi k z} - 2 \sum_{k=0}^{+\infty} i 2 \pi k \tilde \psi_j(k) e^{i 2 \pi (k-1) z} + \sum_{k=0}^{+\infty} i 2 \pi k \tilde \psi_j(k) e^{i 2 \pi (k-2) z} \\
& \quad = \sum_{ k = 2^{j-2} - 2 }^{2^{j+2}} i 2 \pi \big( (k+2) \tilde \psi_j(k+2) - 2 (k+1) \tilde \psi_j(k+1) + k \tilde \psi_j(k) \big) e^{i 2 \pi k z}
  \end{split}
\end{equation*}
Then we bound, for $2^{j-2} - 2 \leq k \leq 2^{j+2}$,
\begin{equation*}
  \begin{split}
      & |(k+2) \tilde \psi_j(k+2) - 2 (k+1) \tilde \psi_j(k+1) + k \tilde \psi_j(k)| \\
      & \quad \leq 2 \sup_h | 2 \tilde \psi_j'(h) + h \tilde \psi_j''(h) | \leq C \Big( \sup_h | 2^{-j} \tilde \psi_0'(2^{-j} h) | + 2^j \sup_h | 2^{-2j} \tilde \psi_0''(2^{-j} h) | \Big) \leq C 2^{-j} \, .
  \end{split}
\end{equation*}
Reasoning analogously for negative $k$'s, using~\eqref{eq:bound_on_denominator}, we get that, for $j \geq 3$,
\begin{equation*}
    | M_j(z) | \leq \frac{1}{|1 - e^{- i 2 \pi z}|^2} \sum_{2^{j-2} - 2 \leq |k| \leq 2^{j+2} + 2} 2 \pi | (k+2) \tilde \psi_j(|k+2|) - 2 (k+1) \tilde \psi_j(|k+1|) + k \tilde \psi_j(|k|) | \leq \frac{C}{|z|^{2}} \, ,
\end{equation*}
for $0 < |z| \leq \frac{1}{2}$.
Moreover, we have the trivial bound
\begin{equation*}
    |M_j(z)| \leq C 2^{2j} \, , \quad \text{for } j \geq 0 \text{ and } |z| \leq \frac{1}{2} \, .
\end{equation*}
It follows that, for $j \geq 3$,
\begin{equation*}
  \begin{split}
      \| M_j \|_{L^1(\T^1)} & = \int_{\T^1} | M_j(z) | \d z \leq C \int_{\T^1} \min\Big\{ 2^{2j}, \frac{1}{|z|^2} \Big\} \, \d z \\
                            & = C \int_{\{ |z| \leq 2^{-j} \}} 2^{2j} \d z + C \int_{\{ 2^{-j} < |z| \leq \frac{1}{2} \}} \frac{1}{|z|^2} \d z \leq C 2^{j} \, .
  \end{split}
\end{equation*}
Inserting this estimate into~\eqref{eq:derivative_as_convolution}, we get that, for $j \geq 0$ (bounding with a constant in the cases $j = 0,1,2$),
\begin{equation*}
    \| (P_j u)' \|_{L^\infty(\T^1)} \leq \| M_j \|_{L^1(\T^1)} \| P_j u \|_{L^\infty(\T^1)} \leq C 2^{j} \| P_j u \|_{L^\infty(\T^1)} \, ,
\end{equation*}
showing~\eqref{eq:bernstein_inequality}.

On the one hand, by~\eqref{eq:holder_estimate_on_dyadic_block} and the previous inequality, we deduce that
\begin{equation} \label{eq:holder_estimate_on_low_frequencies}
  \begin{split}
      | P_j u(x) - P_j u(y) | & \leq C 2^{j} \| P_j u \|_{L^\infty(\T^1)} |x-y| = C 2^{j-js} 2^{js} \| P_j u \|_{L^\infty(\T^1)} |x-y| \\
                              & \leq C 2^{j(1-s)} \|u\|_{B^s_{\infty,\infty}(\T^1)} |x-y| \, .
  \end{split}
\end{equation}
On the other hand, we also have the bound
\begin{equation} \label{eq:holder_estimate_on_high_frequencies}
    | P_j u(x) - P_j u(y) | \leq 2 \| P_j u \|_{L^\infty(\T^1)} \leq C 2^{-js} \| u \|_{B^s_{\infty,\infty}(\T^1)} \, .
\end{equation}

By periodicity, it suffices to consider $x, y \in \R$ with $0 < |x - y| \leq \frac{1}{2}$. 
We choose $\hat j = \lfloor - \log_2 | x - y | \rfloor$ in~\eqref{eq:splitting_holder_estimate} so that, using~\eqref{eq:holder_estimate_on_low_frequencies} for $j \leq \hat j$ and~\eqref{eq:holder_estimate_on_high_frequencies} for $j > \hat j$, we get that
\begin{equation*}
  \begin{split}
      | u(x) - u(y) | & \leq C\|u\|_{B^s_{\infty,\infty}(\T^1)} \Big( \sum_{ 0 \leq j \leq \hat j} 2^{j(1-s)}  |x-y| + \sum_{ j > \hat j } 2^{-js} \Big) \\
                      & \leq C \| u \|_{B^s_{\infty,\infty}(\T^1)} \Big( \frac{2^{(1-s)(\hat j + 1)} - 1}{2^{(1-s)}-1} | x - y | + \frac{2^{-s(\hat j + 1)}}{1 - 2^{-s}} \Big) \\
                      & \leq C \| u \|_{B^s_{\infty,\infty}(\T^1)} |x-y|^{s} \, ,
  \end{split}
\end{equation*}
where we have used that $2^{\hat j} |x-y| \leq 1 < 2^{\hat j + 1} |x - y|$.

We have shown that
\begin{equation*}
    [u]_{C^{0,s}(\T^1)} \leq C \| u \|_{B^s_{\infty,\infty}(\T^1)} \, .
\end{equation*}
It remains to estimate the $L^\infty$ norm of $u$.
For this, we observe that
\begin{equation*}
    | u(x) - \hat u_0 | = \Big| \int_{\T^1} (u(x) - u(y)) \d y \Big| \leq [u]_{C^{0,s}(\T^1)} \int_{-\frac{1}{2}}^{\frac{1}{2}} |z|^s \d z \leq C [u]_{C^{0,s}(\T^1)} \, ,
\end{equation*}
and
\begin{equation*}
    |\hat u_0| = |P_{-1} u| \leq \| P_{-1} u \|_{L^\infty(\T^1)} \leq C \| u \|_{B^s_{\infty,\infty}(\T^1)} \, .
\end{equation*}
This concludes the proof.
\end{proof}

\subsection*{Acknowledgements}
   G.\ M.\ Coclite, F.\ Maddalena, and G.\ Orlando are members of Gruppo Nazionale per l'Analisi Matematica, la Probabilit\`a e le loro Applicazioni (GNAMPA) of the Istituto Nazionale di Alta Matematica (INdAM).

G.\ M.\ Coclite, F.\ Maddalena, and G.\ Orlando  were partially supported by the Italian Ministry of University and Research under the Programme ``Department of Excellence'' Legge 232/2016 (Grant No. CUP - D93C23000100001).

   S.\ Dipierro acknowledges financial support
by the Australian Research Council Future Fellowship FT230100333 New
perspectives on nonlocal equations.

E.\ Valdinoci acknowledges financial support
by the Australian Laureate Fellowship FL190100081 Minimal surfaces,
free boundaries and partial differential equations.

\printbibliography

\end{document}